\documentclass[oneside]{amsart}                     % onecolumn (standard format)
\pdfoutput=1
\usepackage{graphicx}
\usepackage{geometry}
\usepackage{graphicx}
\usepackage{amsfonts}
\usepackage{amssymb}
\usepackage{amsmath}
\usepackage{tikz}
\usepackage{pinlabel}
\usepackage{dbnsymb}
\usetikzlibrary{knots}

\newtheorem{theorem}{Theorem}[section]
\newtheorem{proposition}[theorem]{Proposition}
\newtheorem{lemma}[theorem]{Lemma}
\newtheorem{question}[theorem]{Question}
\newtheorem{corollary}[theorem]{Corollary}

\theoremstyle{remark}
\newtheorem{definition}[theorem]{Definition}
\newtheorem{example}[theorem]{Example}
\newtheorem{remark}{Remark}[section] 

\newtheorem{claim}{Claim}

\newcommand{\lcm}{\text{lcm}}
\newcommand{\eps}{\epsilon}
\newcommand{\Z}{\mathbb{Z}}

\makeatletter
\@namedef{subjclassname@1991}{2020 Mathematics Subject Classification}
\makeatother

\begin{document}

\title{Knot Graphs and Gromov Hyperbolicity  }

\author{Stanislav  Jabuka}
\address{Department of Mathematics and Statistics, University of Nevada, Reno NV, 89557 USA}
\email{jabuka@unr.edu}                      

\author{Beibei Liu}
\address{School of Mathematics, Georgia Institute of Technology, Atlanta, GA, 30332 USA}
\email{bliu96@gatech.edu}                      

\author{Allison H. Moore}
\address{Department of Mathematics \& Applied Mathematics\\ Virginia Commonwealth University, Richmond, VA 23284, USA}   
\email{moorea14@vcu.edu}   

\thanks{S. J. was partially supported by the Simons Foundation, Award ID 524394, and by the NSF, Grant No. DMS--1906413. B. L. is grateful to Max Planck Institute for Mathematics in Bonn for its hospitality and financial support. A. H. M. was partially supported by the NSF, Grant No. DMS--1716987.}



\maketitle

\begin{abstract}
We define a broad class of graphs that generalize the Gordian graph of knots. These knot graphs take into account unknotting operations, the concordance relation, and equivalence relations generated by knot invariants. 
We prove that overwhelmingly, the knot graphs are not Gromov hyperbolic, with the exception of a particular family of quotient knot graphs. 
We also investigate the property of homogeneity, and prove that the concordance knot graph is homogeneous. Finally, we prove that that for any $n$, there exists a knot $K$ such that the ball of radius $n$ in the Gordian graph centered at $K$ contains no connected sum of torus knots.
\end{abstract}

%%%%%

\section{Introduction}
\label{SectionIntroduction}

The Gordian graph is a countably infinite graph in which each vertex represents the isotopy type of a knot, and two vertices are connected by an edge whenever the corresponding knots are related by a crossing change. 
A variant of this graph can be similarly defined given any unknotting operation, and any such `knot graph' may be regarded as a geodesic metric space with the usual distance metric on a graph. 
Although unknotting number, or more generally, the H(n)-unknotting numbers, are widely studied knot invariants, the general structure of these graphs remains mysterious.  
The main aim of this article is to study their global metric properties. We will prove:
\begin{theorem} \label{MainTheorem1}
\label{short not hyp}
The Gordian graph, the $H(n)$-Gordian graph for $n\geq 2$, and the concordance Gordian graph are not Gromov hyperbolic. 
\end{theorem}
A geodesic metric space is Gromov hyperbolic, or $\delta$-hyperbolic, if every geodesic triangle is $\delta$-thin for some $\delta \ge 0$. A geodesic triangle is $\delta$-thin whenever each edge is contained in the closed $\delta$-neighborhood of the union of the remaining two. Our strategy in the proof of Theorem \ref{short not hyp} is the direct construction of geodesic triangles that are never $\delta$-thin. % for any $\delta$.  
In contrast to Theorem \ref{short not hyp}, we prove:
\begin{theorem} \label{TheoremMain2}
\label{short hyp}
The quotients of the Gordian graph induced by the smooth four-genus, unknotting number, Heegaard Floer $\tau$-invariant and Khovanov homology $s$-invariant, and the quotient of the $H(2)$-Gordian graph induced by the non-orientable smooth four-ball genus are all isometric to a subspace of $\mathbb{Z}$. In particular, they are all Gromov hyperbolic.
\end{theorem}

To state Theorem \ref{short not hyp} and Theorem \ref{short hyp} more precisely (see Theorems \ref{TheoremAboutNonHyperbolicityResults} and \ref{TheoremAboutHyperbolicityResults}), we introduce the general definition of a {\em knot graph} $\mathcal K_{\mathcal O}$ with respect to an unknotting operation $\mathcal{O}$ in Definition \ref{DefinitionOfKnotGraphs} and extend this definition to quotients of knot graphs  induced by knot invariants or under equivalence generated by concordance. In particular, the concordance knot graph $\mathcal{CK}_{\mathcal{O}}$ associated to an unknotting operation $\mathcal{O}$ is the graph whose vertices are concordance classes of oriented knots, in which a pair of vertices span an edge if there exist oriented knots representing those classes related by an $\mathcal{O}$-move. 

To the best of our knowledge, Definition \ref{DefinitionOfKnotGraphs} is sufficiently general to include all instances of knot graphs that have appeared in the literature thus far. The \emph{Gordian graph} $\mathcal K_{\slashoverback}$ (where $\slashoverback$ indicates the crossing change operation) has been studied for instance in \cite{Baader2}, \cite{Baader3},  \cite{Baader1},  \cite{Blair}, \cite{GambaudoGhys}, \cite{HirasawaUchida}. 
The {\em band-Gordian graph} $\mathcal K_{H(2)}$ and its analogues, the {\em $H(n)$-Gordian graphs} for $n\ge 3$, have been considered in \cite{ZhangYang}, \cite{ZhangYang2018}, and the {\em pass-move Gordian complex} $\mathcal K_{\#}$ appears in \cite{NakanishiOhyama}. We are not aware of any previous results about the concordance knot graphs $\mathcal{CK}_\mathcal O$, but \cite{IchiharaJong} studies a quotient graph where an equivalence relation on edges is induced by the Conway polynomial and the unknotting operation is the pass-move. 

We remark that Gambaudo and Ghys \cite{GambaudoGhys} previously established a quasi-isometric embedding of the integer lattice $\mathbb{Z}^d$ into the set of knots with a metric equivalent to the edge-metric on the knot graph $\mathcal{K_{\slashoverback}}$, i.e. the Gordian graph.  They noted the naturality of this metric in the sense that the Gordian distance between two knots is the minimum number of generic double points over immersed homotopies relating them. Their construction explicitly involves torus knots. This raises the question as to whether a genuine quasi-isometry could be constructed via torus knots. We prove this is not the case. 

Let $B_{r}(v)$ denote the radius $r$ ball  centered on the vertex $v$ in the Gordian  graph. 
\begin{theorem} \label{TheoremMain3}
\label{torus ball}
For all $r>0$, there exists a knot $K$ such that $B_{r}(K)$ does not contain any arbitrary connected sum of torus knots.
\end{theorem}

Besides the hyperbolicity of the knot graphs, we  also study the property of homogeneity. A metric space $(X,d)$ is {\em homogeneous} if for every $x,y\in X$ there exists an isometry $\psi:X\to X$ with $\psi(x) = y$, i.e. if the isometry group of $X$ acts transitively on $X$. In Section \ref{homogeneity}, we show

\begin{theorem} \label{TheoremMain4}
\label{homogeneous}
The concordance graph associated with any set of unknotting operations is always homogeneous. The quotients of the Gordian graph $Q\mathcal K^\tau _{\slashoverback}$ and $Q\mathcal K^{s'}_{\slashoverback}$ with respect to the $\tau$ and $s$-invariants are homogeneous. 
\end{theorem}

We  pose several other questions about the structure of the knot graphs, and study the link of the class in the unknot in several quotient knot graphs in Section \ref{homogeneity}.

%%%%%%%%%%%%%%%%%%%%%%%%%%%%%%%%%%%%%%%%%%%%%%%%%%%%%%%%%%%%%
%%%%%%%%%%%%%%%%%%%%%%%%%%%%%%%%%%%%%%%%%%%%%%%%%%%%%%%%%%%%%

\subsection{Organization}
Section \ref{SectionGeodesicMetricSpaces} provides a range of background material, including discussions on metrics on graphs and geodesics in the resulting metric spaces, Gromov hyperbolicity and quasi-isometries, definitions of general knot graphs and quotients of knot graphs, bounds on the distance function in $H(n)$-Gordian graphs, and computations of first homology groups of certain Brieskorn spheres. 
In Section \ref{Hn and Concordance triangles} we construct explicit geodesic triangles in the $H(n)$-Gordian knot graphs and in the  concordance knot graphs, that are not $\delta$-thin for any $\delta\ge 0$. Section \ref{quotient triangles} is devoted to the study of quotient knot graphs, and two general theorem are established that in some cases completely identify their isometry type (Theorems \ref{IsometryBetweenQuotinetKnotGraphAndIntegers} and \ref{IsometryBetweenQuotinetKnotGraphAndIntegersCaseTwo}).  
Lastly, Section \ref{proofs} provides the proofs of Theorems \ref{MainTheorem1}--\ref{TheoremMain4} in Sections \ref{SubsectionProofOfMainTheorem1}--\ref{homogeneity} respectively. Before proving them,  some theorems are restated there in greater generality first. 
%%%%%%%%%%%%%%%%%%%%%%%%%%%%%%%%%%%%%%%
%%%%%%%%%%%%%%%%%%%%%%%%%%%%%%%%%%%%%%%%
%%%%%%%%%%%%%%%%%%%%%%%%%%%%%%%%%%%%%%%%
\section{Background Material} 
\label{SectionGeodesicMetricSpaces}
This section provides a panoply of background material upon which the proofs of Theorems \ref{MainTheorem1}--\ref{TheoremMain4} are based. Sections \ref{geodesic metric spaces}--\ref{quasi and hyp} review material pertaining to graphs as metric spaces and their hyperbolicity properties. Section \ref{knot graphs} defines general knot graphs, of which the examples appearing in Theorems \ref{MainTheorem1}--\ref{TheoremMain4} are special cases. Sections \ref{sec:H(n)moves} and \ref{sec:bounds} remind the reader of the unknotting moves $H(n)$ that generalize noncoherent band moves when $n\ge 3$, and give some bounds on the associated distance function $d_n$ between knots. Lastly, Section \ref{brieskorn} provides background on 3-dimensional Brieskorn spheres, including computations of the first homology group of some examples.   
%%%%%%%%%%%%%%%%
%%%%%%%%%%%%%%%%
%%%%%%%%%%%%%%%%

\subsection{Geodesic Metric Spaces}
\label{geodesic metric spaces}

Let $(X,d)$ be a metric space, and let $\alpha:[a,b]\to X$ be a path. Given a partition $\mathcal P=\{t_0, \dots, t_n\}$ of $[a,b]$, let 

\[ L(\alpha, \mathcal P) = \sum _{i=1}^n d(\alpha (t_{i-1}), \alpha (t_i)) \]

denote the {\em polygonal length of $\alpha$ asssociated to the partition $\mathcal P$}. We say that $\alpha$ is a {\em rectifiable path} if the supremum of its polygonal lengths, taken over all partitions of $[a,b]$, is finite. In that case we define the {\em length $L(\alpha)$ of $\alpha$} as said supremum:
\[ L(\alpha) = \sup_{\mathcal P} L(\alpha, \mathcal P). \]

It is easy to check that if $\alpha :[a,b]\to X$ is a rectifiable path, then so is its restriction to any segment $[c,d]\subset [a,b]$.

A metric space $(X,d)$ is called a {\em geodesic metric space} if for every pair of points $x,y\in X$ there exists a rectifiable path $\alpha:[0,1]\to X$ with $\alpha(0) = x$, $\alpha (1) = y$ and with 
\[ L(\alpha |_{[s,t]}) = d(\alpha (s), \alpha (t)), \qquad \qquad  \forall s, t \in [0,1]. \]

Any such path $\alpha$ is called a {\em geodesic path}. A {\em geodesic triangle} $\{\alpha, \beta, \gamma\}$ in a geodesic metric space $(X,d)$ is a triple of geodesics $\alpha, \beta, \gamma :[0,1]\to X$ with $\alpha (1) = \beta (0)$, $\beta (1) = \gamma(0)$ and $\gamma(1) = \alpha (0)$. We refer to $\alpha$, $\beta$ and $\gamma$ (or sometimes their images in $X$) as the {\em edges} of the geodesic triangle, and the points $\{\alpha(1), \beta(1), \gamma(1)\}$ as the {\em vertices} of the geodesic triangle.

For  $\delta\ge 0$, a geodesic triangle $\{\alpha_1, \alpha_2, \alpha_3\}$ is called {\em $\delta$-thin} if for every $i\in \{1,2,3\}$ and every $x\in \text{Im}(\alpha_i)$, the inequality 
\[ d(x, \cup _{j\ne i} \text{Im}(\alpha_j))\le \delta \]
holds. 
%%%
%%%
\begin{definition} \label{DefinitionOfGromovHyperbolic}
The geodesic metric space $(X,d)$ is called {\em $\delta$-hyperbolic} if every geodesic triangle in $X$ is $\delta$-thin, and we say that $(X,d)$ is {\em Gromov hyperbolic} if it is $\delta$-hyperbolic for at least one $\delta\ge 0$. 	
\end{definition}
%%%
%%%
Observe that if $X$ is $\delta$-hyperbolic then it is also $\delta'$-hyperbolic for every $\delta'\ge \delta$. If $X$ is Gromov hyperbolic, we let 
\[ \delta(X) = \inf \{\delta \ge 0\,|\, X \text{ is $\delta$-hyperbolic}\}. \]

%%%%%%%%%%%%%%%%%%%%%%%%%%%%%%%%%%%%%%%%%%%%%%%%%%%
%%%%%%%%%%%%%%%%%%%%%%%%%%%%%%%%%%%%%%%%%%%%%%%%%%%
%%%%%%%%%%%%%%%%%%%%%%%%%%%%%%%%%%%%%%%%%%%%%%%%%%%

\subsection{Graphs as Geodesic Metric Spaces}
\label{graphs as metric}

Let $G$ be a graph and let $Vert(G)$ and $Edge(G)$ denote its sets of vertices and edges respectively. A graph $G$ can be viewed as a 1-dimensional CW complex whose 0-cells are the vertices of $G$, and whose 1-cells are in one-to-one correspondence with the edges of $G$. Specifically, for each edge $e\in Edge(G)$ with endpoints $v, w\in Vert(G)$ we attach a 1-cell $\alpha_e\cong [0,1]$ to $Vert(G)$ whose attaching map identifies the two endpoints $\{0,1\}$ of the 1-cell $\alpha_e$ with $v$ and $w$. This endows the graph $G$ with the structure of a topological space, in such a way that $G$ is connected as a graph if and only if it is path-connected as a topological space. 

We next define a metric $d$ on a connected graph $G$, by first defining it for vertices $v, w \in Vert(G)$ as:
\[ d(v,w) = \text{ Minimum number of edges needed to connect $v$ to $w$}. \]
Note that this definition tacitly gives each edge in the graph length 1. The distance between a pair of points $x, y$ lying on the same 1-cell $\alpha_e\cong [0,1]$ is 
\[
d(x,y) = \left\{
\begin{array}{cl}
|x-y| & \quad ; \quad \text{ if }  \alpha _e \text{ has two distinct endpoints}, \cr

\min \{|x-y|, |x|+|1-y| \} & \quad ; \quad \text{ if } \alpha _e \text{ has only one endpoint}. 
\end{array}
\right.
\]

In the above definition, we assume in the second case that the attaching map takes the unit interval to a circle of radius $1/2\pi$ so that $d(x, y)$ just corresponds with the distance along a circle of circumference one.  
Lastly, given points $x, y$ lying on distinct edges $\alpha_e$ and $\alpha_{e'}$ with boundary vertices $\{v_0, v_1\}$ and $\{w_0, w_1\}$ respectively, we define their distance $d(x,y)$ as 
\[
d(x,y) = \min_{i,j \in \{0,1\}}  d(x,v_i)+d(v_i, w_j)+d(w_j,y).
\]
With these definitions in place, it is now easy to verify that for a connected graph $G$, the pair $(G,d)$ becomes a geodesic metric space. 
We shall use this structure on graphs implicitly on all knot graphs in subsequent sections. 
%%%%%%%%%%%%%%%%%%%%%%%%%%%%%%%%%%%
%%%%%%%%%%%%%%%%%%%%%%%%%%%%%%%%%%%
\subsection{Quasi-isometries and hyperbolicity}
\label{quasi and hyp}

A map $f: X_1\rightarrow X_2$ between metric spaces $(X_1, d_{1})$, $(X_2, d_{2})$ is called a \emph{quasi-isometric embedding} if there are constants $a\geq 1, b\geq 0$ such that the double inequality
\[ \dfrac{1}{a} d_{1}(x, x')-b\leq d_{2}(f(x), f(x'))\leq ad_{1}(x, x')+b, \]
holds for all $x, x'\in X_1$. In addition, if there is a constant $C\geq 0$ such that for every $y\in Y$ there exists an $x\in X$ with 
\[ d_{2}(y, f(x))\leq C,\] 
then $X_1$ and $X_2$ are called \emph{quasi-isometric}. If $C=0$, the map $f$ is called \emph{bi-Lipschitz}. 

Gromov hyperbolicity is invariant under quasi-isometries between geodesic spaces. 

\begin{proposition}\cite[Theorem 12]{Ghy}
\label{hyperbo}
Assume that  $(X_1, d_{1})$ is quasi-isometric to $(X_2, d_{2})$ with parameters $a, b$ and $C$. If $X_{1}$ is $\delta$-hyperbolic, then $X_{2}$ is $\delta'$-hyperbolic with $\delta'$ depending on $\delta, a, b, C$. 
\end{proposition}

\begin{corollary}
If $(X_1, d_1)$ and $(X_2, d_2)$ are quasi-isometric geodesic metric spaces and $X_1$ is not $\delta$-hyperbolic for any $\delta \ge 0$, then neither is $X_2$. 
\end{corollary}

\begin{remark}
An interesting result by Bowditch \cite{Bowditch} (see also Chapter 6 in \cite{Gromov} as well as \cite{PortillaRodriguezTouris})   
posits that hyperbolicity of a geodesic metric space is equivalent to the hyperbolicity of a graph associated to it, underscoring the ``approximately-tree-like" nature of hyperbolic spaces. 
This result puts the onus on understanding and exploring hyperbolicity in graphs, which is partially the motivation for this work. 
\end{remark}
%%%%%%%%%%%%%%%%%%%%%%%%%%%%%%%%%%%
%%%%%%%%%%%%%%%%%%%%%%%%%%%%%%%%%%%

\subsection{Knot Graphs}
\label{knot graphs}
An {\em unknotting operation $\mathcal O$} on knot diagrams is a local modification/move on a knot diagram, with the property that any knot diagram can be unknotted with a finite number of such $\mathcal O$-moves and/or their inverses. Examples of unknotting operations abound and include the crossing change operation and the infinite family of $H(n)$-moves, $n\ge 2$ from Figure \ref{FigureOfH(n)Moves}. 

%%%%%%
%%%%%%
\begin{figure}[h] 
\centering
\includegraphics[width=8cm]{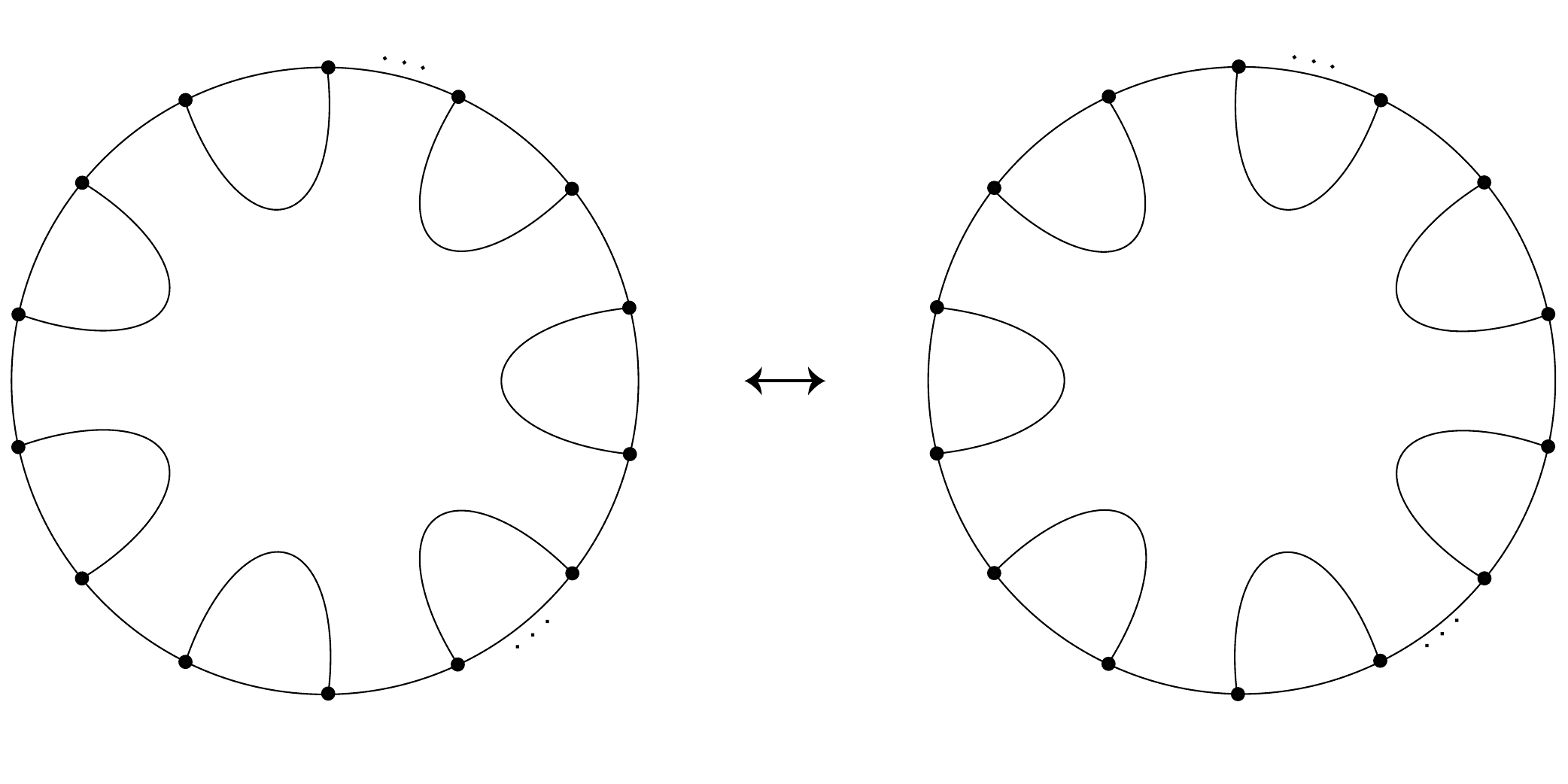}
\put(-185,-10){(a)}
\put(-50,-15){(b)}
\caption{The $H(n)$-move is the unknotting operation which replaces the pattern from subfigure (a) in a knot diagram, with the pattern from subfigure (b). The operation is to be performed so as to preserve the number of components. Shown here is the example of $n=7$. These moves were first introduced and studied by Hoste, Nakanishi and Taniyama \cite{HosteNakanishiTaniyama}.} 
\label{FigureOfH(n)Moves}
\end{figure}
%%%
%%%
\begin{definition} \label{DefinitionOfKnotGraphs}
Let $\mathcal O$ be an unknotting operation on knot diagrams, and for $m\in \mathbb N$ let $\mathcal I_1, \dots, \mathcal  I_m$ be integer-valued knot invariants and let $\mathbb I = \{\mathcal I_1, \dots, \mathcal I_m\}$.  
\begin{itemize}
\item[(i)] The {\em $\mathcal O$-Gordian Knot Graph} $\mathcal K_{\mathcal O}$ associated to the unknotting operation $\mathcal O$ is the graph whose vertices are unoriented knots, and in which a pair of knots $K$ and $K'$ span an edge if they possess diagrams related by an $\mathcal O$-move (or its inverse). 
\vskip2mm
%%%
\item[(ii)] The {\em Concordance Knot Graph} $\mathcal{CK}_{\mathcal O}$ associated to the unknotting operation $\mathcal O$ is the graph whose vertices are concordance classes $[K]$ of oriented knots $K$, and in which a pair of concordance classes $[K]$ and $[K']$ span an edge if there exist oriented knots $L$ and $L'$ concordant to $K$ and $K'$ respectively, and such that $L$ and $L'$ possess diagrams related by an $\mathcal O$-move (or its inverse). 
\vskip2mm
%%%
\item[(iii)] The {\em Quotient Knot Graph} $Q\mathcal{K}_{\mathcal O}^{\mathbb I}$ associated to the unknotting operation $\mathcal O$ and the collection of knot invariants $\mathbb I$ is the graph whose vertices are equivalence classes $[K]_\mathcal O^\mathbb I$ of knots $K$, by which a pair of knots $K$ and $K'$ are equivalent if $\mathcal I_i(K) = \mathcal I_i(K')$ for all $i=1, \dots, m$. Two equivalence classes $[K]_\mathcal O^\mathbb I$ and $[K']_\mathcal O^\mathbb I$ span an edge if there exist knots $L$ and $L'$ equivalent to $K$ and $K'$ respectively, and such that $L$ and $L'$ possess diagrams related by an $\mathcal O$-move (or its inverse). 
\end{itemize}
We shall collectively refer to these 3 types of graphs as {\em Knot Graphs}. 
\end{definition} 
%%%
\begin{remark}
In part (iii) of the definition above, we assume that all invariants $\mathcal I_i$ in the collection $\mathbb I$ are preserved under orientation reversal.
\end{remark}
%%%
\begin{remark}
Following \cite{HirasawaUchida} one can define the structure of a simplicial complex on all the knot graphs, by letting a collection of $n+1$ vertices span an $n$-simplex if each pair of vertices spans an edge. This leads to very rich simplicial structures on the knot graphs. For example, it has been shown that in the knot graphs $\mathcal K_{\slashoverback}$, $\mathcal K_{H(n)}$, $\mathcal K_\#$ ($\#$ = the pass move), $\mathcal K_{RCC}$ (``RCC" = Region Crossing Change) each edge of the graph lies in an $n$-simplex for any $n\in\mathbb{N}$, cf.\cite{HirasawaUchida}, \cite{ZhangYang}, \cite{ZhangYang2018}, \cite{GillEtAl} respectively.
\end{remark}

We note that the knot graphs from Definition \ref{DefinitionOfKnotGraphs} can be generalized still by allowing multiple unknotting operations $\mathcal O_1, \dots, \mathcal O_k$ to be considered simultaneously. In such knot graphs, vertices share an edge if they posses representative knots that are related by an $\mathcal O_i$-move (or its inverse) for at least one $i\in \{1, \dots, k\}$. If we let $\mathbb O = \{\mathcal O_1,\dots, \mathcal O_m\}$, we denote the resulting knot graphs by $\mathcal K_{\mathbb O}$ or $Q\mathcal{K}_{\mathbb O}$. An instance of this type of graph has been studied in \cite{Yoshiyuki}. 
\vskip3mm
Different choices of unknotting operations $\mathcal O$ and knot invariants $\mathcal I$ lead to infinitely many examples of knot graphs. While it would be desirable to understand hyperbolicity properties of these general knot graphs, presently existing techniques place a limit on what can be proved. We therefore restrict our considerations on what we perceive as the most important examples. These rely on the principal unknotting operations studied in knot theory, namely the 

\[
\text{Unknotting Operations}   = \left\{  
\begin{array}{rcl}
\slashoverback & = & \text{ Crossing change operation.} \cr
H(2) & = &\text{ The non-coherent (or non-orientable)} \cr
&  &  \text{ band move.}
\end{array} \right. 
\]

As indicated by our choice of notation, the {\em non-coherent} or {\em non-orientable} band move corresponds to the $H(2)$-move from Figure \ref{FigureOfH(n)Moves}. As some of our results readily generalize from the $H(2)$-move to the $H(n)$-moves for all $n\ge 2$, we shall consider the latter unknotting operations as well. The knot graphs in Parts (i) and (ii) from Definition \ref{DefinitionOfKnotGraphs} are fully determined by the choice of one of these unknotting operations.  

To motivate our choice of knot invariants used in the construction of the knot graphs from Part (iii) of Definition \ref{DefinitionOfKnotGraphs}, we first make this definition.

\begin{definition} \label{DefinitionOfCompatibility}
An unknotting operation $\mathcal O$ and an integer valued knot invariant $\mathcal I$ are said to be {\em compatible} if changing a knot $K$ by a single $\mathcal O$-move (or its inverse) changes $\mathcal I(K)$ by at most 1. Said differently, if $K$ and $K'$ are knots related by an $\mathcal O$-move or its inverse, then  $|\mathcal I(K) - \mathcal I(K')|\le 1$.  A quotient graph $Q\mathcal K^\mathbb I_\mathcal O$, with $\mathbb I = \{\mathcal I_1, \dots, \mathcal I_m\}$  is said to be {\em compatible} if $\mathcal O$ is compatible with every $\mathcal I_j \in \mathbb I$.  
\end{definition}

We have found that knot graphs $\mathcal K_\mathcal O^\mathbb I$ that are not compatible, are rather difficult to understand, and some exhibit rather surprising properties (see Example \ref{ExampleForNonCompatibleCase}). Accordingly, having chosen our unknotting operations to be the crossing-change operation and the $H(n)$-moves, we were compelled to pick knot invariants from among those compatible with said unknotting operations. Specifically, we consider these knot invariants:
\[
\text{Knot Invariants} =\left\{  
\begin{array}{rl}
g_4 & =  \text{ Orientable smooth 4-genus} \cr
\gamma_4 & = \text{ Non-orientable smooth 4-genus} \cr
u & = \text{ Unknotting number} \cr
\tau & = \text{ Ozsv\'ath-Szab\'o's tau invariant} \cr
s & = \text{ Rasmussen's $s$ invariant}
\end{array} \right. 
\]

Of these, $g_4, u, \tau, s/2$ are compatible with the crossing-change operation, while $\gamma_4$ is compatible with non-coherent band moves.

%%%%%%%%%%%%%%%%
%%%%%%%%%%%%%%%%

\subsection{ $H(n)$-moves}
\label{sec:H(n)moves}

The $H(n)$-move, $n\ge 2$, is defined in Figure \ref{FigureOfH(n)Moves}. We adopt the convention from \cite{HosteNakanishiTaniyama} that only those $H(n)$-moves are allowed which preserve the number of components. The $H(2)$-move is called a {\em noncoherent} or {\em nonorientable band move}, as it is realized by attaching a band to the knot, in such a way that the orientation of the band agrees with that of the knot at one of its ends, and disagrees at the other. 

The $H(n)$-moves were introduced and studied by Hoste, Nakanishi and Taniyama in \cite{HosteNakanishiTaniyama}, where they proved that each $H(n)$-move is an unknotting operation. We are thus justified in letting $\mathcal K_{H(n)}$ denote the resulting $H(n)$-Gordian knot graphs, and we denote the induced metric on $\mathcal K_{H(n)}$ by $d_n$. Hoste, Nakanishi and Taniyama  established several estimates for the $H(n)$-unknotting number $u_n(K)$, defined as $d_n(K,U)$ (with $U$ the unknot). The following theorem is proved in \cite{HosteNakanishiTaniyama} for the case of $K'=U$, we adapt their proofs for our somewhat more general formulas. 

\begin{theorem}[Hoste, Nakanishi, Taniyama \cite{HosteNakanishiTaniyama}] \label{TheoremHosteNakanishiTaniyama}
Let $K, K'$ be a pair of knots and $n\ge 2$ an integer.
\begin{itemize}
\item[(i)] An $H(n)$-move can be realized by an $H(n+1)$-move. In particular

\[ d_{n}(K, K')\geq d_{n+1}(K, K').\]  

\item[(ii)] $\lim_{n\rightarrow \infty} d_{n}(K, K')=1$.
\item[(iii)]  If $n\ge 3$ then $(n-1)d_{n}(K, K')\geq \frac{2}{3} d_{2}(K, K').$
\end{itemize}
\end{theorem}
%%%
%%%
\begin{proof}
(i) The fact that an $H(n)$-move can be realized as an $H(n+1)$-move is shown in Lemma 2 and Figure 10 in \cite{HosteNakanishiTaniyama}. From this the inequality $d_n(K,K') \ge d_{n+1}(K,K')$ is obvious. 
\vskip1mm
\noindent (ii) For $K'=U$ this formula is the content of Theorem 6 in \cite{HosteNakanishiTaniyama}. We modify the proof of the said theorem to obtain the claimed result. Each $H(n)$-move can be obtained by a sequence of $(n-1)$ $H(2)$-moves, each of which is realized by attaching a noncoherent band as in Figure \ref{FigureOfH(n)MovesRealizedByH(2)Moves}. 
%%%
%%%
\begin{figure} 
\centering
\includegraphics[width=9cm]{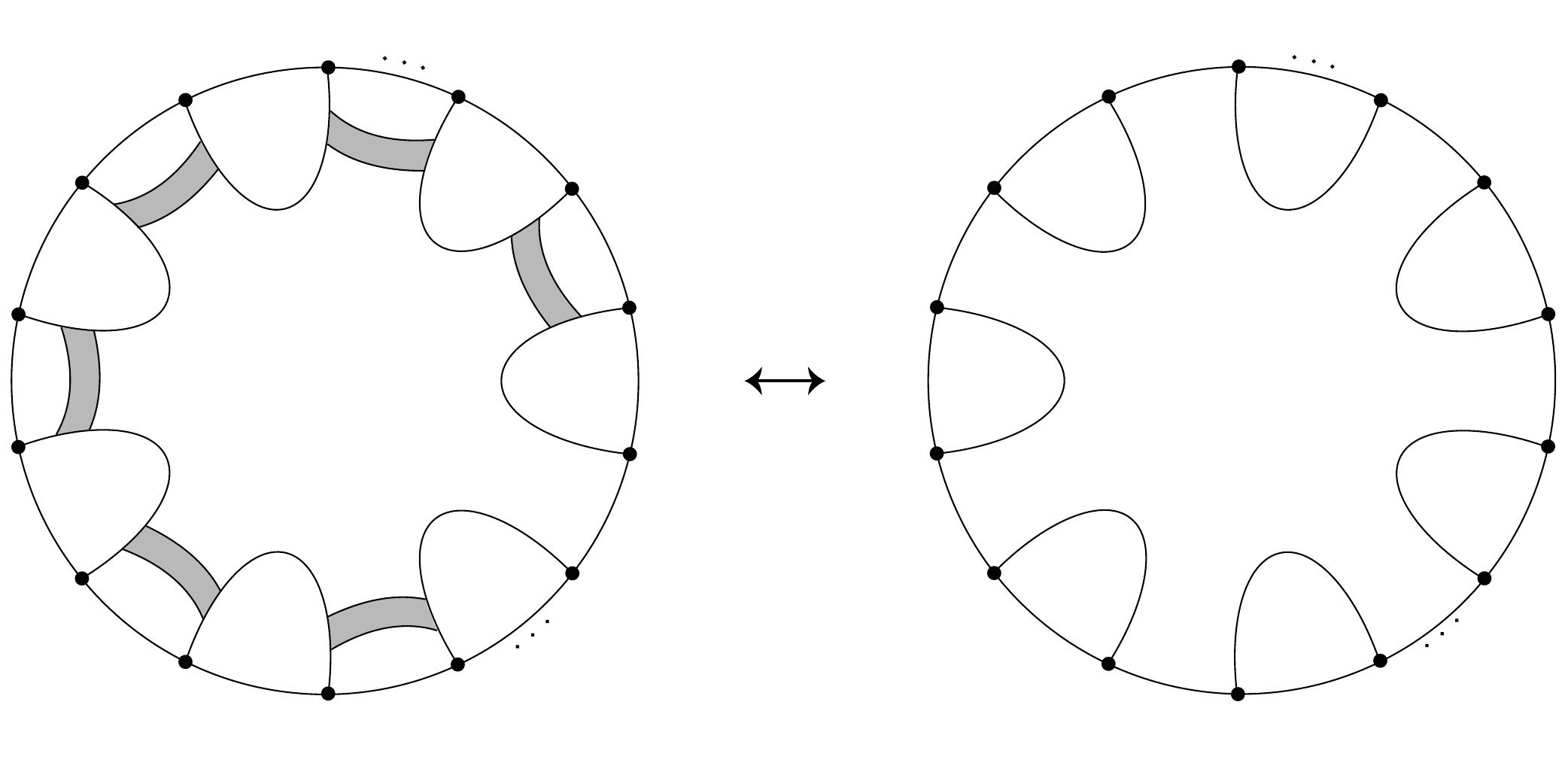}
\caption{An $H(n)$-move can be realized by $n-1$ $H(2)$-moves, i.e. by attaching $n-1$ noncoherent bands, as indicated. Pictured here is the case of $n=7$.  } \label{FigureOfH(n)MovesRealizedByH(2)Moves}
\end{figure}
%%%
%%%
Let $n\ge 2$ be arbitrary. Since one can pass from a diagram for $K$ to a diagram for $K'$ by applying $d_{n}(K, K')$ $H(n)$-moves, it follows that the diagrams of $K$ and $K'$ are related by $(n-1)\cdot d_n(K,K')$ noncoherent band attachments. By sliding bands if necessary, we may assume that all the bands are disjoint. Furthermore, we may gather the root of each band near one point of the knot $K'$ as in Figure \ref{Pic5HnConvergesTo1}. 
%%%
%%%
\begin{figure} 
\centering
\includegraphics[width=10cm]{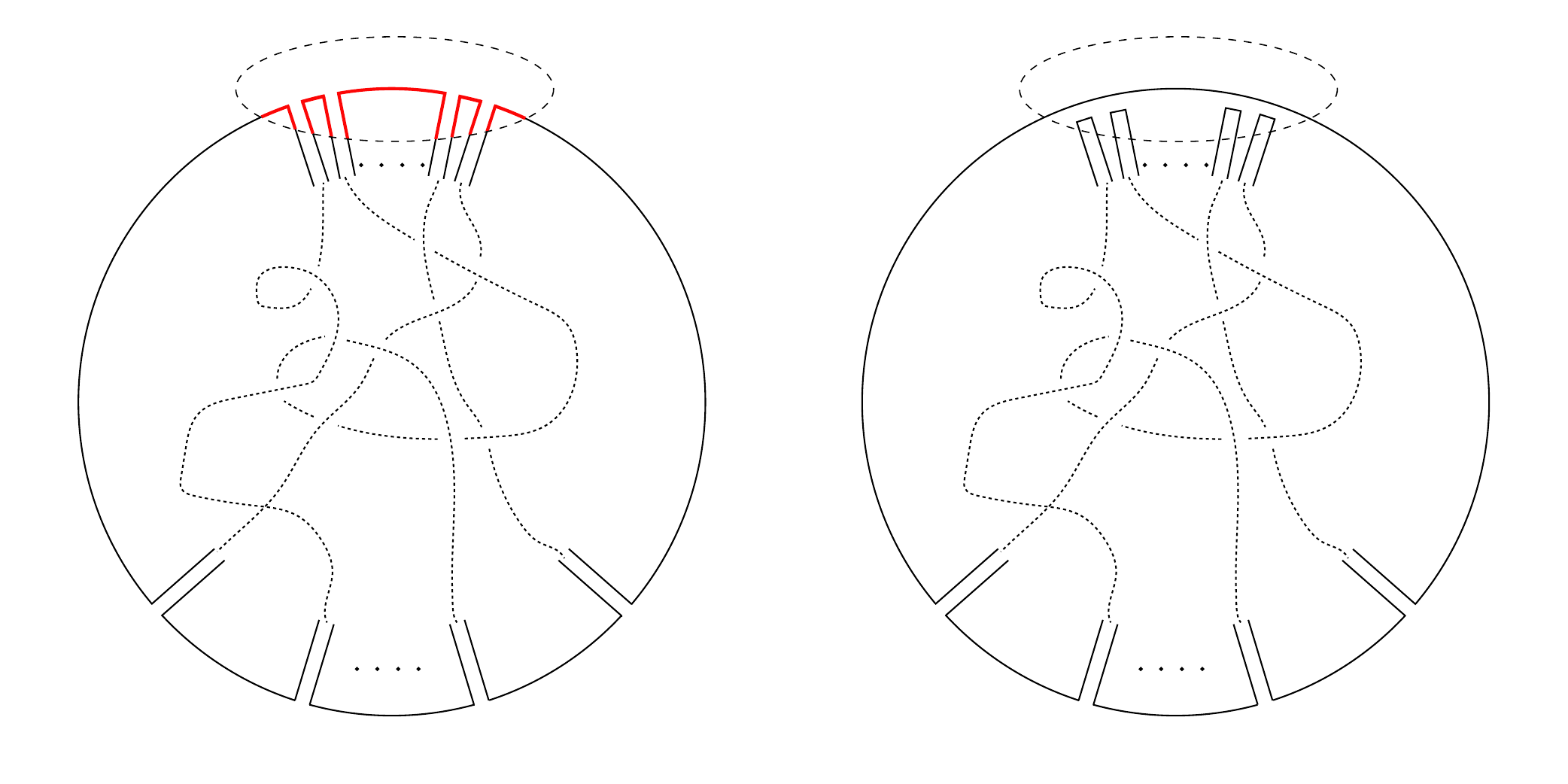}
\put(-220,-10){(a)} 
\put(-75,-10){(b)} 
\put(-290,80){$K'$} 
\put(-10,80){$K$}
\caption{If a knot $K'$ can be obtained from $K$ by $m$ $H(n)$-moves, then $K'$ can also be obtained from $K$ by $m(n-1)$ $H(2)$-moves, each of which is realized by the attaching of a noncoherent band. If the roots of the bands are gathered as shown, the totality of all $m(n-1)$ band moves is accomplished by a single $H(m(n-1)+1)$-move, the one inside the dashed oval. Observe that the number of arcs (colored in red in Figure (a)) inside the dashed oval equals 1 plus the number of bands.  } \label{Pic5HnConvergesTo1}
\end{figure}
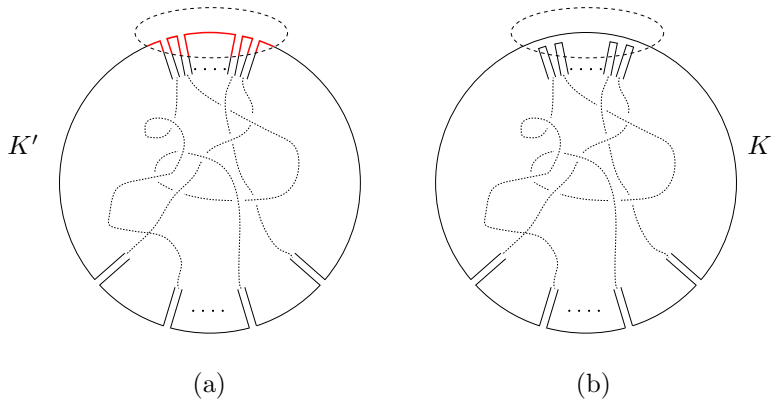
%%%
%%%
It is now an easy observation that all $(n-1)d_n(K,K')$ noncoherent band moves are realized by a single $H((n-1)d_n(K,K')+1)$-move (see again Figure \ref{Pic5HnConvergesTo1}), showing that 
\[ d_{(n-1)d_n(K,K')+1}(K,K') = 1. \]

The proof of Part (ii) of the present theorem follows from this formula and Part (i). 
\vskip1mm
\noindent (iii) This formula for $K'=U$ is Part (6) of Theorem 7 in \cite{HosteNakanishiTaniyama}, and the proof is readily adapted for our purposes. Let $m\ge 2$ be such that $d_{m}(K, K')=1$ (such an $m$ exists by Part (ii), e.g. $m=(n-1)d_n(K,K')+1$ for any $n\ge 2$). Then the diagrams of $K$ and $K'$ can be related by a single $H(m)$-move, and hence also by $m-1$ band attachments. Individually, each of these $m-1$ band attachments may be component preserving or may change the number of components by one. In particular, cutting one band either yields another knot or a two-component link. Each band of the first kind can be removed by a single $H(2)$-move. Consider then a band of the second variety, and specifically consider an \lq\lq inner-most\rq\rq one, i.e. a band whose roots divide the knot into  two arcs, each of which contains at most one root of any other band. It must be then that there exists a pair of bands as in Figure \ref{Pic6SpecialBands}. 
%%%
%%%
\begin{figure} 
\centering
\includegraphics[width=5cm]{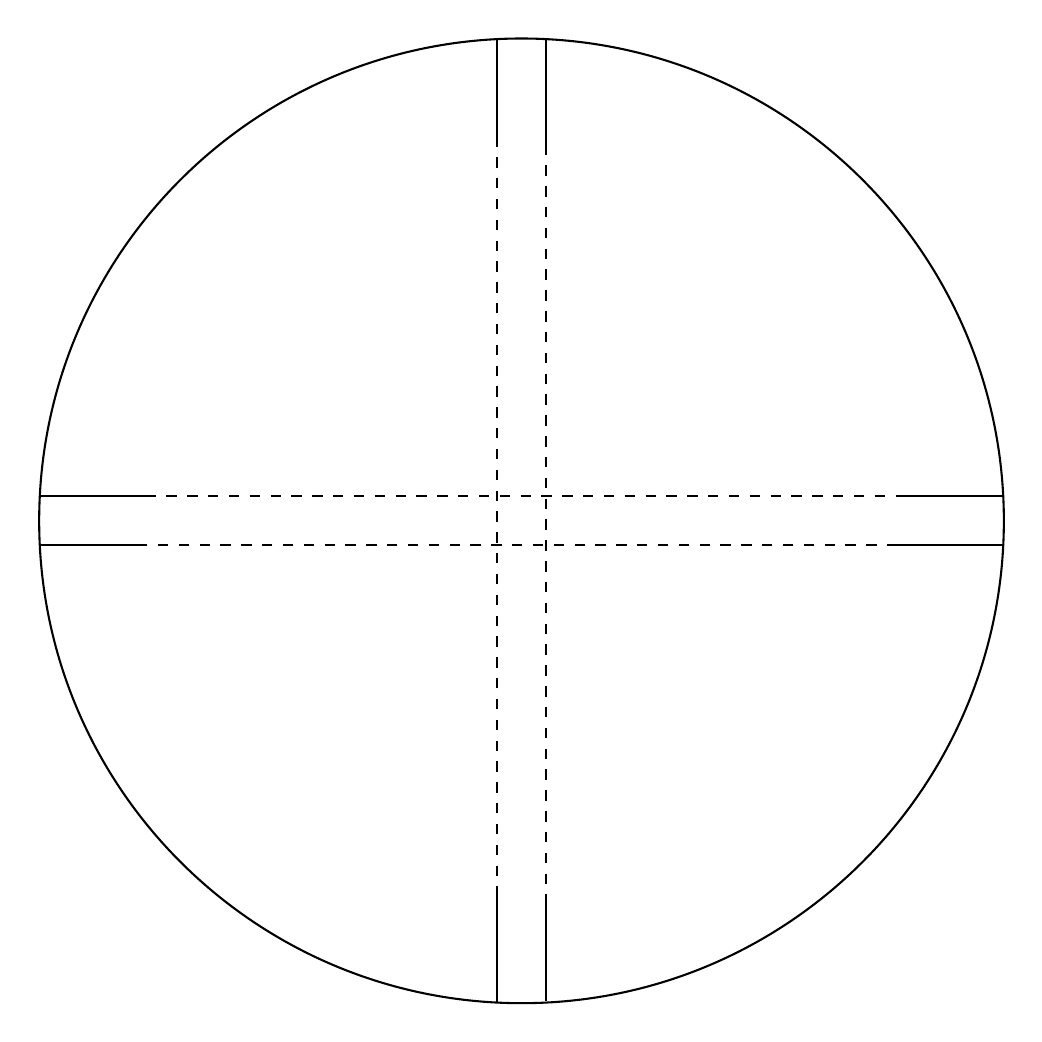}
\caption{A special pair of bands used to related the diagrams of $K$ and $K'$. } \label{Pic6SpecialBands}
\end{figure}
%%%
%%%    
It is shown in \cite{HosteNakanishiTaniyama}, Figure 15, that this pair of bands can be removed by 3 $H(2)$-moves. Therefore, we can remove noncoherent bands with either a single $H(2)$-move, or we can remove pairs of noncoherent bands with 3 $H(2)$-moves. Since removing all $m-1$ noncoherent bands changes a diagram for $K$ to one for $K'$, we obtain the inequality 
\[ \frac{3}{2}(m-1)\ge d_2(K,K'). \]
The claim now follows since we may pick $m=(n-1)d_n(K,K')+1$ for any $n\ge 3$. 
\end{proof}
%%%
%%%
A direct and important consequence of the preceding theorem is this observation. 

\begin{corollary}
\label{quasigordian}
The $H(n)$-Gordian graph $\mathcal K_{H(n)}$ for $n\ge 3$ is %$(3(n-1)/2, 0)$ 
quasi-isometric to the $H(2)$-Gordian graph $\mathcal K_{H(2)}$. 
\end{corollary}

\begin{figure} 
\centering
\includegraphics[width=5.5cm]{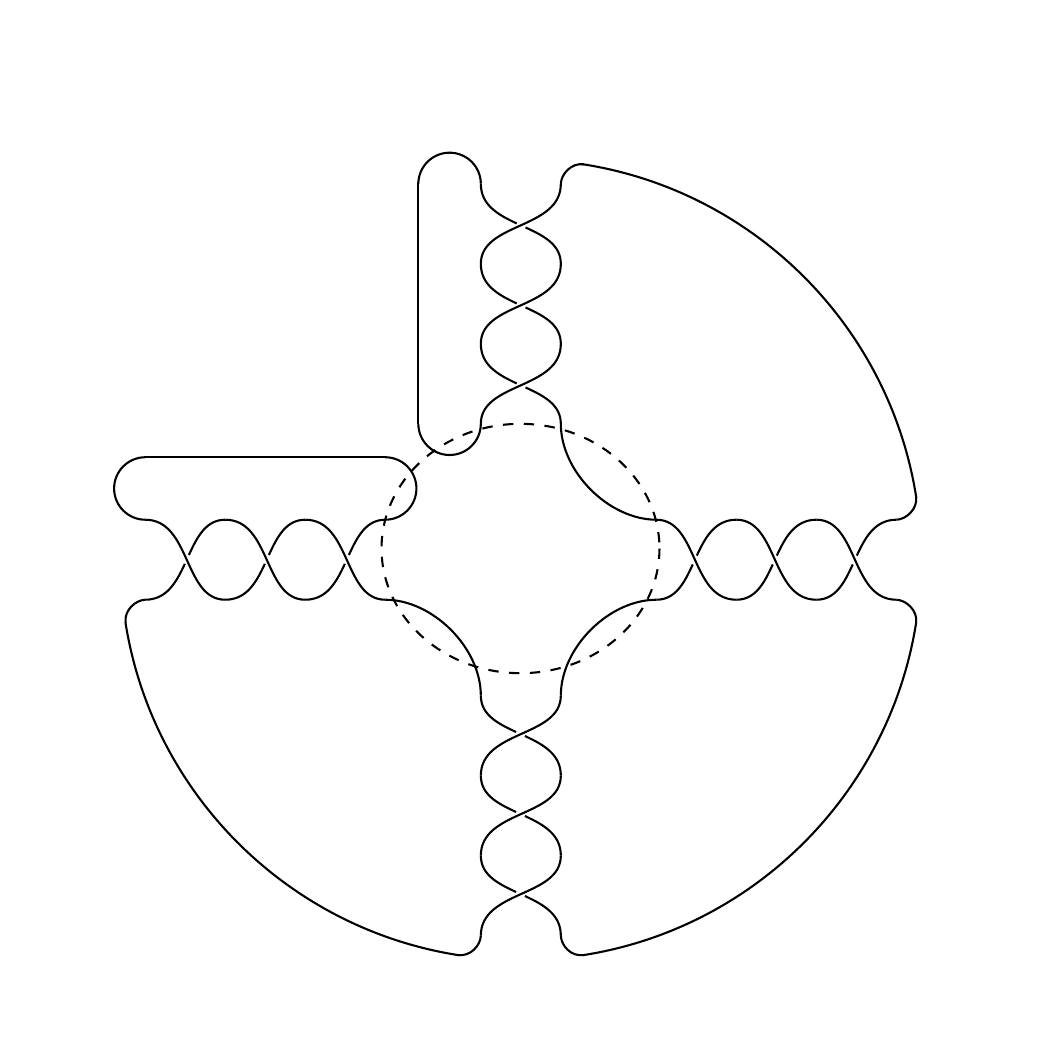}
\caption{A 4-fold connected sum of $T(3,2)$ with itself, can be unknotted by the single $H(5)$-move indicated in the dashed oval. A generalization of this picutre shows that the connected sum $\#^{n-1}T(2,k)$ for any odd $k$ and $n\ge 2$, can be unknotted by a single $H(n)$-move. } \label{Pic7HnMovesOnConnnectedSumsOfTorusKnots}
\end{figure}

%%%%%%%%%%%%%%%%%%%%%%%
%%%%%%%%%%%%%%%%%%%%%%%
%%%%%%%%%%%%%%%%%%%%%%%

\subsection{Bounds on $d_n$ coming from cyclic branched covers}
\label{sec:bounds}

For a knot $L$, let $\Sigma_m (L)$ denote the closed 3-manifold obtained as the $m$-fold cyclic cover of $S^3$ with branching set the knot $L$. If $m=2$ we simply write $\Sigma (L)$ instead of $\Sigma _2(L)$. Let $e_m(L)$ denote the minimum number of generators of $H_1(\Sigma_m(L);\Z)$ and let $e_m^p(L)$ denote the minimum number of generators of $H_1(\Sigma_m(L);\Z_p)$ (here and elsewhere $\mathbb Z_p$ denotes the cyclic group on $p$ elements).

\begin{proposition}\cite[Theorem 4]{HosteNakanishiTaniyama} \label{PropositionBounds}
\label{hnt}
Let $K$ and $K'$ be a pair of knots and let $m, n \ge 2$, then 
\[ \frac{|e_m(K)-e_m(K')|}{(n-1)(m-1)} \leq d_n(K,K'). \]
\end{proposition}
%%%
%%%
\begin{proof}
The case of $K'=U$ is Theorem 4 in \cite{HosteNakanishiTaniyama}, and we use its proof as a starting point for the proof presented here. Suppose a single $H(n)$-move changes a knot $L$ to a knot $L'$. Then $\Sigma _m(L)$ and $\Sigma _m(L')$ are related by a surgery on a handlebody of genus $(m-1)(n-1)$ (as the $m$-fold branched cover of a \lq\lq thickened disk\rq\rq containing the loops as in Figure \ref{FigureOfH(n)Moves} is a handlebody of said genus).  Lemma 3 from \cite{HosteNakanishiTaniyama} shows that in this case the inequality $|e_m(L)-e_m(L')| \le (m-1)(n-1)$ holds. 

If $K'$ is obtained from $K$ by $d_n(K,K')$ $H(n)$-moves, with intermediate knots 
\[
	K=K_0 \to K_1 \to \dots \to K_{d_n(K,K')}=K',
\] 
then 
\[
	|e_m(K)-e_m(K')| \le \sum _{i=1}^{d_n(K,K')} |e_m(K_{i-1}) - e_m(K_i)| \le d_n(K,K')(m-1)(n-1),
\]
proving the claim. 
\end{proof}

Recall that the invariants $e_2^3(K)$ and $e_2^5(K)$ are the minimum number of generators of the first homology of the branched double cover with coefficients in $\mathbb{Z}/3$ and $\mathbb{Z}/5$, respectively. As it turns out, these invariants are related to certain evaluations of the Jones and Q-polynomials (see \cite{LickorishMillett}, \cite{Jones-Q}). In \cite{AK}, Abe-Kanenobu give criteria for knots to be related by an $H(2)$-move in terms of evaluations of these polynomials. This can in turn be rephrased as a lower bound in terms of the invariants $e_2^3(K)$ and $e_2^5(K)$ as follows:

\begin{proposition}\cite[Corollary 5.6, Corollary 5.10]{AK} 
\label{CorollaryAboutObstructions}
Let $K$ and $K'$ be a pair knots. Then
\begin{enumerate}
	\item $|e_2^3(K)-e_2^3(K')|\leq d_2(K, K')$, and 
	\vskip2mm
	\item $|e_2^5(K)-e_2^5(K')|\leq d_2(K, K')$. 
\end{enumerate}
\end{proposition}

\subsection{Homology of Brieskorn manifolds}
\label{brieskorn}
In this section, we review the algorithm for computing the homology groups of Brieskorn manifolds, following Orlik \cite{Orlik}. 

For integers $w_1, w_2, w_3 >1$,the Brieskorn manifold $\Sigma(w_1, w_2, w_3)$ is defined as 
\[
	\{(z_1, z_2, z_3)\in \mathbb C^3\,|\, z_1^{w_1} +z_2^{w_1} +z_3^{w_1}=0\} \cap S_{\varepsilon} 
\]
where $S_\eps$ is a sphere in $\mathbb C^3$ centered at $0$ and of radius $\eps>0$ chosen sufficiently small  so that the only singularity contained inside of $S_\eps$ is at $z=0$. Brieskorn manifolds are closed, oriented 3-manifolds, and by Milnor \cite{Milnor}, they can also be obtained as an $r$-fold cyclic  cover of $S^{3}$ with the branching set  the torus knot or link $T(p, q)$. 

To determine both the rank and the torsion subgroup of  $H_{1}(\Sigma(w_{1}, w_{2}, w_{3}); \mathbb{Z})$ we first proceed to define several sets of numbers following \cite{Orlik} (see also the reference \cite{Randell}), specialized here to varieties of only three variables $\{z_{1}, z_{2}, z_{3}\}$. For any ordered subset $I$ of $\{1,2,3\}$ let $\kappa(I)$ be defined as 
\[
	\kappa(I)=\sum_{J\subseteq I} (-1)^{|J|-|I|} \dfrac{\prod w_J}{ \lcm (w_J)}, 
\]   
where $w_J \subseteq \{ w_1, w_2, w_3 \}$ is the subset corresponding with indexing subset $J \subseteq  \{1,2,3\}$.
In the definition of $\kappa(I)$ we adopt the convention that $\Pi w_\emptyset / \lcm (w_\emptyset)$ equals 1. Furthermore, let 
\[
    \kappa'(I)=\left\{
                \begin{array}{ll}
                \kappa(I) & \text{ when } |I| \text{ is even, and}\\
				0 & \text{ when } |I| \text{ is odd.}
                \end{array}
              \right.
\]
Next, define inductively on $|I|$ the numbers $c(I)$ as $c(\emptyset) =1$ and  
\[
	c(I)=\frac{ \gcd (\{w_1, w_2, w_3\} - w_I)}{\prod_{J\subset I} c(J)}.
\]
Lastly, let $r=\max_{I\subseteq \{1,2,3\}} \kappa(I)$, and for  $j=1,\dots, r$ define $d_j$ as 
\[
	d_{j}=\prod_{I | \kappa'(I)  \geq j > 0} c(I).
\]
\begin{proposition}\cite[Propositions 2.6 and 3.4]{Orlik}
\label{rank and torsion}
The rank of $H_{1}(\Sigma(w_{1}, w_{2}, w_{3}); \mathbb{Z})$ equals $\kappa(\{1,2,3\})$ and the torsion subgroup of $H_{1}(\Sigma(w_{1}, w_{2}, w_{3});\mathbb Z)$ is isomorphic to 
\[
	\Z_{d_{1}}\oplus \Z_{d_{2}}\oplus \cdots \oplus \Z_{d_{r}}.
\]
\end{proposition}

\begin{lemma}
\label{comp}
Let $k$ be an odd natural number, then 
\[
	H_{1}(\Sigma(2, k, k); \Z)\cong \left(  \Z_{2}\right)^{k-1}.
\]
\end{lemma}

\begin{proof}
By Proposition \ref{rank and torsion} the rank of $H_{1}(\Sigma(2, k, k); \Z)$ is zero. A direct computation of $\kappa' (I)$ shows that 
\[
\kappa'(I)=\left\{
    \begin{array}{ll}
    1 & I=\emptyset \\
    k-1 & I=\{2,3\} \\
	0 & \text{ otherwise.}
    \end{array}
  \right.          
\]
The only values of $c(I)$ appearing in the calculation of the $d_j$ are therefore $c(\emptyset)=1$ and $c(\{2,3\})=2$. Thus $r=k-1$ and for all $j=1,\dots, k-1$ we obtain 
\[
	d_j = \prod_{\kappa(I)  \geq j} c(I) = c(\emptyset) \cdot c(\{2,3\}) = 2. \qedhere
\]
\end{proof}
\begin{example} \label{BrieskornSphereSigma(2,15,9)}
Consider the Brieskorn manifold $\Sigma (2,15,9)$. Its rank is easily seen to equal zero. Moreover, $\kappa'(\emptyset) =1$, $\kappa'(\{2,3\}) =2$ and $\kappa'(I)=0$ otherwise. It follows that $r=2$ and $d_j = c(\emptyset)\cdot c(\{2,3\}) = 2$ for all $j=1,2$. We conclude that 
\[
	H_1(\Sigma (2,15,9);\mathbb Z) \cong \left( \mathbb Z_2\right)^2.
\]
\end{example}
\begin{example} \label{BrieskornSphereSigma(2,15,5)}
The rank of $\Sigma (2,15,5)$ is similarly equal zero. Here, $\kappa'(\emptyset) =1$, $\kappa'(\{2,3\}) =4$ and $\kappa'(I)=0$ otherwise. It follows that $r=4$ and $d_j = c(\emptyset)\cdot c(\{2,3\}) = 2$ for all $j=1,2,3,4$. We conclude that 
\[
	H_1(\Sigma (2,15,5);\mathbb Z) \cong \left( \mathbb Z_2\right)^4.
\]

\end{example} 
%%%%%%%%%%%%%%%%%%%%%%%%%%%%%%%%%%%%%%%%%%%%%%%%%%%%%%%%%%%%%
%%%%%%%%%%%%%%%%%%%%%%%%%%%%%%%%%%%%%%%%%%%%%%%%%%%%%%%%%%%%%
\section{Geodesic Triangles in Knot Graphs}
\label{Hn and Concordance triangles}
Recall that a metric space is Gromov hyperbolic if it is $\delta$-hypberbolic for at least one $\delta\ge0$ (cf. Definition \ref{DefinitionOfGromovHyperbolic}). Accordingly, one proves the absence of Gromov hyperbolicity in a metric space by showing that for every $\delta\ge 0$ there exists a geodesic triangle that is not $\delta$-thin. We prove such a statement for the case of the knot graphs $\mathcal K_{H(n)}$ (Proposition \ref{PropositionOnDeltaHyperbolicityOfHnGordianGraphs} in Section \ref{Hn triangles}) and the concordance graph $\mathcal{CK}_{\slashoverback}$ (Proposition \ref{PropositionOnDeltaHyperbolicityOfConcordanceGraphs} in Section \ref{Concordance triangles}). These results are then used in Section \ref{proofs} to prove Theorem \ref{MainTheorem1} (see specifically the proof of Theorem \ref{TheoremAboutNonHyperbolicityResults}). 
%%%%%%%%%%%%%%%%%%%%%%%%%%%%%%
%%%%%%%%%%%%%%%%%%%%%%%%%%%%%%
\subsection{ $H(n)$-Gordian Knot Graphs}
\label{Hn triangles}
Consider any three mutually distinct knots $K_0, K_1, K_2$. Theorem \ref{TheoremHosteNakanishiTaniyama} implies that there exists an $n_0\in \mathbb N$ such that for all $n\ge n_0$ the equality $d_n(K_i,K_j) = 1$ holds for any pair of distinct indicies $i,j \in \{0,1,2\}$. Accordingly, the knots $K_0, K_1, K_2$ are the vertices of a geodesic triangle in $\mathcal K_{H(n)}$ and this triangle is plainly $\delta$-hyperbolic for all $\delta \ge 1/2$. In contrast to this conclusion, we will show that with $n\ge 2$ fixed, and for any $\delta \ge 0$, there exist geodesic triangles in $\mathcal K_{H(n)}$ which are not $\delta$-hyperbolic. We begin with the case of $n=2$. 

For $m\in \mathbb N$, whose value is to be determined later, consider the vertices given by the unknot $U$, the knot $K_{1}=\#^{m}T(2, 9)$ and the knot $K_{2}=\#^{m}T(2, 9)\#(\#^{m}T(2, 15))$  in $\mathcal K_{H(2)}$. Observe that each of $T(2,9)$ and $T(2,15)$ can be unknotted by a single $H(2)$-move, and hence $d_2(U,K_1)\le m$ and $d_2(U, K_2) \le 2m$. Lower bounds on $d_2(U,K_1)$ and $d_2(U, K_2)$ come courtesy of Proposition \ref{PropositionBounds}. Indeed, since $\Sigma (T(2,9))\cong L(2,9)$ and $\Sigma(T(2,15))\cong L(2,15)$, then  
$$H_1(\Sigma(K_1);\mathbb Z) \cong \left(\mathbb Z_9\right)^m \quad \text{ and } \quad  H_1(\Sigma(K_2);\mathbb Z) \cong \left(\mathbb Z_9\right)^m\oplus \left(\mathbb Z_{15}\right)^m .$$
It follows that $e_2(K_1) = m$ and $e_2(K_2)=2m$ and therefore that $d_2(U, K_1)\ge m$ and $d_2(U,K_2)\ge 2m$, implying that $d_2(U,K_1) = m$ and $d_2(U, K_2) = 2m$. This shows that the path $\ell_1$ connecting $U$ to $K_1$ through the edges with vertices $\#^iT(2,9)$, $i=1, \dots, m-1$ is a geodesic path. The same is true of the path $\ell_3$ connecting $U$ to $K_2$ via the edges in $\mathcal K_{H(2)}$ with the vertices 
\begin{align*}
T(2, 15) \leftrightarrow   T(2, 9)\# T(2, 15)\leftrightarrow   T(2, 9)\# (\#^{2} T(2, 15))\leftrightarrow   (\#^{2}T(2, 9)) \#(\#^{2} T(2, 15))\leftrightarrow \cr
\dots \leftrightarrow  (\#^{m-1} T(2, 9))\#(\#^{m-1} T(2, 15))\leftrightarrow   (\#^{m-1} T(2, 9))\#(\#^{m} T(2, 15)).
\end{align*}

By a similar argument, it follows that $d_2(K_{1}, K_{2})=m$, and the path $\ell_2$ connecting $K_{1}$ to $K_{2}$ passing through the knots $(\#^{m} T(2, 9)) \#(\#^{i} T(2, 15))$, $i=1, \dots, m-1$, is a geodesic. 

Next we shall estimate from below the distance of a particular vertex $K_3$ from the edge $\ell_3$, to the union $\ell_1\cup \ell_2$ in the geodesic triangle constructed above. 

\begin{claim}

Assume that $m=2k$ is even. Set $K_3 = (\#^{k} T(2, 9))\#(\#^{k}T(2, 15))$, and observe that $K_3$ is a vertex on the path $\ell_3$. Then 
\[ d_{2}(K_3, \ell_{1}\cup \ell_{2})\geq 3k/4. \]
\end{claim}
\proof We first prove that $d_{2}(K_{3}, \ell_{1})\geq k$. Let $K=\#^iT(2,9)$ be a vertex on the path $\ell_1$, then by Corollary \ref{CorollaryAboutObstructions} we obtain
\[ d_{2}(K_{3}, K)\geq |e^{5}_{2}(K_{3})-e^{5}_{2}(K)| =k. \]
This concludes that $d_2(K_3, \ell_1)\ge k$.

Now we let $K$ denote a knot on the path $\ell_2$ of the form $(\#^{2k} T(2, 9))\#( \#^{i} T(2, 15))$ where $i\in\{0, \cdots , 2k\}$. Consider the $9$-fold   cyclic covers of $S^{3}$ with branching sets  $T(2, 9)$ and $T(2, 15)$, respectively. These are the Brieskorn manfolds $\Sigma(2, 9, 9)$ and $\Sigma(2, 15, 9)$ respectively, and Lemma \ref{comp} and Example \ref{BrieskornSphereSigma(2,15,9)}  imply that 
\[ H_{1}(\Sigma(2, 9, 9); \Z)\cong \left( \Z_{2}\right)^8 \quad \text{ and } \quad H_{1}(\Sigma(2, 15, 9); \Z)\cong \left( \Z_{2} \right)^2. \]
Applying Corollary \ref{CorollaryAboutObstructions} we obtain 
\[ d_{2}(K_{3}, K)\geq \dfrac{|e_{9}(K)-e_{9}(K_{3})|}{8}=\dfrac{6k+2i}{8}\geq \dfrac{3k}{4}. \]
This implies that $d_2(K_3,\ell_2)\ge 3k/4$ and thus that $d_2(K_3,\ell_1\cup \ell_2)\ge 3k/4$. Since $K_3$ lies on $\ell_3$, we conclude that $\ell_3$ is not contained in the closed $\delta$-neighborhood of $\ell_1\cup \ell_2$ whenever $3k/4>\delta$. 
\qed

Given any $\delta \ge 0$, choose $k$ in the above construction so that $3k/4>\delta$, then the geodesic triangle $\{\ell_1, \ell_2, \ell_3\}$ is not $\delta$-thin. It follows that $\mathcal K_{H(2)}$ is not $\delta$-hyperbolic for any $\delta\ge 0$. 

By Corollary \ref{quasigordian}, $\mathcal K_{H(n)}$ is quasi-isometric to $\mathcal K_{H(2)}$ for any $n>2$. 
Corollary \ref{hyperbo} now implies that $\mathcal K_{H(n)}$ is also not $\delta$-hyperbolic for any $\delta\ge 0$. This proves Part (i) of Theorem \ref{TheoremAboutNonHyperbolicityResults}. 

\vskip3mm
We conclude this section with an explicit construction of a geodesic triangle in $\mathcal K_{H(n)}$ that can be used to disprove its $\delta$-hyperbolicity directly. Indeed, the construction of said triangle proceeds in analogy with the case $n=2$ given above. The needed modifications are replacing $T(2, 9)$ by $\#^{n-1} T(2, 9)$ and replacing $T(2, 15)$ by $\#^{n-1} T(2, 15)$. We consider then in $\mathcal K_{H(n)}$ the vertices 
\[
	U, \quad K_1 =\#^{m(n-1)} T(2, 9) \quad \text{ and }\quad K_{2}=(\#^{m(n-1)} T(2, 9)) \# (\#^{m(n-1)} T(2, 15)).
\]
As before, $m\ge 1$ is an integer whose value is to be determined later. Since each of $\#^{n-1}T(2,9)$ and $\#^{n-1}T(2,15)$ can be unknotted by a single $H(n)$-move (as in Figure \ref{Pic7HnMovesOnConnnectedSumsOfTorusKnots}), it follows that $d_n(U,K_1)\le m$ and $d_n(U,K_2) \le 2m$. An application of Proposition \ref{PropositionBounds} shows that $d_n(U,K_1)\ge m$ and $d_n(U,K_2)\ge 2m$, leading to $d_n(U,K_1) = m$ and $d_n(U,K_2)=2m$. It follows that the path $\ell_1$ connecting $U$ to $K_1$ by passing through the vertices $\#^{i(n-1)}T(2,9)$, $i=1,\dots, m-1$ is a geodesic path in $\mathcal K_{H(n)}$. The same is true of the path $\ell_3$ connecting $U$ to $K_2$ by passing through the vertices 
\begin{align*}
\#^{(n-1)}T(2, 15) \leftrightarrow  (\#^{(n-1)}T(2, 9))\# (\#^{(n-1)}T(2, 15))  \cr
\leftrightarrow (\#^{(n-1)} T(2, 9))\# (\#^{2(n-1)} T(2, 15)) \leftrightarrow (\#^{2(n-1)}T(2, 9)) \#(\#^{2(n-1)} T(2, 15)) \cr
\leftrightarrow  \dots \leftrightarrow (\#^{(m-1)(n-1)} T(2, 9))\#(\#^{(m-1)(n-1)} T(2, 15)) \cr
\leftrightarrow (\#^{(m-1)(n-1)} T(2, 9))\#(\#^{m(n-1)} T(2, 15)).
\end{align*}
Lastly let $\ell_2$ be the path connecting $K_1$ to $K_2$ via edges in $\mathcal K_{H(n)}$ with intermediate vertices $(\#^{m(n-1)}T(2,9)) \#(\#^{i(n-1)}T(2,15))$, $i=1,\dots, m-1$. Clearly $d_n(K_1, K_2)\le m$ (to see this use $m$ $H(n)$-moves to unknot the $m$ summands of $\#^{n-1}T(2,15)$ in $K_2$) while an application of Proposition \ref{PropositionBounds} gives $d_n(K_1, K_2) \ge m$. This shows that $d_n(K_1, K_2) = m$ and that $\ell_2$ is a geodesic path.

As before, let $m=2k$  with $k$ to be chosen later, and let 
\[ K_3 =(\#^{k(n-1)}T(2,9))\#(\#^{k(n-1)}T(2,15)). \]
Note that $K_3$ is a vertex on the path $\ell_3$, and as before we obtain the inequality  
\[ d_{n} (K_3, \ell_{1}\cup \ell_{2}) \geq 3k/4. \]
If $K=\#^{i(n-1)}T(2,9)$ is any vertex on the path $\ell_1$, then Proposition \ref{PropositionBounds} implies 
\[
	d_{n} (K_{3}, K)\geq \frac{|e_{5}(K_{3})-e_{5}(K)|}{ 4(n-1)}  = \frac{|4(n-1)k - 0| }{4(n-1)} = k.
\]
The above calculation use the facts that $\Sigma_5(T(2,9) \cong \Sigma (2,9,5)$ which is an integral homology sphere, and $\Sigma _5(T(2,15))\cong \Sigma (2,15,5)$. According to Example \ref{BrieskornSphereSigma(2,15,5)} we obtain $H_1(\Sigma (2,15,5);\mathbb Z) \cong (\mathbb Z_2)^4$. 
Thus $d_{n}(K_{3}, \ell_{1})\geq k$. 

Next, we consider $d_{n}(K_{3}, \ell_{2})$. Let $K$ be a vertex on the path $\ell_2$. Then  $K$ is of the form $(\#^{2k(n-1)} T(2, 9)) \#(\#^{i(n-1)}T(2, 15))$ where $0\leq i\leq 2k$. Consider again the 9-fold cyclic covers of $S^{3}$ with the branching sets $T(2, 9)$ and $T(2, 15)$ as before. Then 
\[
	d_{n}(K_{3}, K) \geq \dfrac{|e_{9}(K)-e_{9}(K_{3})|}{8(n-1)}\geq \dfrac{6(n-1)k+2i(n-1)}{8(n-1)}\geq \dfrac{3k}{4}.
\]
Given any $\delta \ge 0$, pick $k$ as large enough so that $3k/4>\delta$, then we have constructed a geodesic triangle $\{\ell_1, \ell_2, \ell_3\}$ that is not $\delta$-thin. In summary, we proved:

\begin{proposition} \label{PropositionOnDeltaHyperbolicityOfHnGordianGraphs}
For every $n\ge 2$ and every $\delta\ge 0$ there exists a geodesic triangle in the knot graph $\mathcal K_{H(n)}$ that is not $\delta$-thin. Accordingly, $\mathcal K_{H(n)}$ is not $\delta$-hyperbolic for any $\delta \ge 0$. 	
\end{proposition}

%%%%%%%%%%%%%%%%%%%%%%%%%%%%%%%%%%%%%%%%%%%%%%%%%%%%%%%%%%%%%
\subsection{Concordance Knot Graphs}
\label{Concordance triangles}

For simplicity of notation let $s'(K) = \frac{1}{2}s(K)$, where $s(K)$ is the Rasmussen invariant of the knot $K$ \cite{Ras}. We shall still refer to $s'$ itself as the Rasmussen invariant. Let $\tau (K)$ denote the Ozsv\'ath-Szab\'o concordance tau invariant \cite{OZS1}. Then  the distance function $d$ on $\mathcal{CK}_{\slashoverback}$ satisfies the lower bounds 
\[
	d([K],[K']) \ge |\tau (K) - \tau (K')|\quad \text{ and } \quad d([K],[K']) \ge |s' (K) - s' (K')|,
\]
whenever the knot $K$ and $K'$ have diagrams that differ by a single crossing change \cite{OZS1}, \cite{Ras}. 

A result of Hedden--Ording \cite{HeddenOrding} stipulates that the knot $K_{0,1}:=D_+(T(2,3),2)$ (the 2-twisted positive Whitehead double of the right-handed trefoil knot $T(2,3)$) has Ozsv\'ath-Szab\'o and Rasmussen invariants given by 
\[ \tau(K_{0,1}) = 0 \quad \text{ and } \quad s'(K_{0,1}) = 1. \]
Since all Whitehead doubles of nontrivial knots have unknotting number equal to 1, we obtain $u(K_{0,1})=1$. Let $K_{1,1} = -T(2,3)$ and observe that 
\[ \tau(K_{1,1}) = -1, \quad s'(K_{1,1}) = 1, \quad \text{ and } \quad u(K_{1,1}) = 1. \]
Pick an integer $k\in \mathbb N$, and form a triangle with edges $\ell_1$, $\ell_2$, $\ell_3$ constructed as follows: 
\begin{itemize}
\item The edge $\ell _1$ connects the class of the class of the unknot $[U]$ to the class of the knot $[\#^{k} K_{0,1}]$ with intermediate vertices given by $[\#^m K_{0,1}]$, $m=1, ,\dots, k-1$. Let $L:=\#^kK_{0,1}$. 
\item The edge $\ell_2$ connects $[L]=[\#^{k} K_{0,1}]$ to $[L\#(\#^{k}K_{1,1})]$ with intermediate vertices given by $[L\#(\#^mK_{1,1})]$, $m=1, ,\dots, k-1$.
\item The edge $\ell_3$ connects $[U]$ to $[L\#(\#^{k}K_{1,1})]$ with intermediate vertices given by 
\[
\begin{array}{c}
 [U ]\leftrightarrow [K_{1,1}]\leftrightarrow[\#^2K_{1,1}]\leftrightarrow[\#^3K_{1,1}]\leftrightarrow\dots \leftrightarrow [\#^kK_{1,1}]\leftrightarrow \cr \cr
[\#^kK_{1,1}]\leftrightarrow[K_{0,1}\#(\#^kK_{1,1})]\leftrightarrow[(\#^2K_{0,1})\#(\#^kK_{1,1})]\leftrightarrow\dots \leftrightarrow [(\#^kK_{0,1})\#(\#^kK_{1,1})].
\end{array}
\]
\end{itemize}
We first show that all three edges are geodesic paths in $\mathcal{CK}_{\slashoverback}$. 

Pick a pair of vertices $[\#^mK_{0,1}]$ and $[\#^n K_{0,1}]$ in $\ell_1$, with $m,n\in \{0,\dots, k\}$. Then $\#^mK_{0,1}$ and $\#^n K_{0,1}$ are related by $|m-n|$ crossing changes, showing that $d([\#^mK_{0,1}], [\#^n K_{0,1}])\le |m-n|$. On the other hand
\[ d([\#^mK_{0,1}], [\#^n K_{0,1}]) \ge |s'(\#^mK_{0,1})-s'(\#^nK_{0,1})| = |m-n|, \]
showing that $d([\#^mK_{0,1}], [\#^n K_{0,1}]) = |m-n|$ and thus that $\ell_1$ is a geodesic edge. 

Similarly, pick a pair of vertices  $[L\#(\#^mK_{1,1})]$ and $[L\#(\#^n K_{1,1})]$ in $\ell_2$, with $m,n\in \{0,\dots, k\}$. Then $L\#(\#^mK_{1,1})$ and $L\#(\#^n K_{1,1})$ are related by $|m-n|$ crossing changes, showing that 
\[ d([L\#(\#^mK_{1,1})], [L\#(\#^n K_{1,1})])\le |m-n|.\] 
On the other hand
\[
	d([L\#(\#^mK_{1,1})], [L\#(\#^n K_{1,1})]) \ge |\tau(L\#(\#^mK_{1,1}))-\tau(L\#(\#^nK_{1,1}))| = |m-n|,
\]
showing that $d([L\#(\#^mK_{1,1})], [L\#(\#^n K_{1,1})])  = |m-n|$ and thus that $\ell_2$ is a geodesic edge. 

Lastly, consider two vertices from $\ell_3$. Since the value of $s'$ is increasing by exactly 1 as we pass from the starting vertex $[U]$ of $\ell_3$ towards the final vertex $[L\#(\#^kK_{1,1})]$ of $\ell_3$, and since every pair of neighboring vertices in $\ell_3$ are related by a crossing change, a similar argument applies here too, showing $\ell_1\cup\ell_2\cup \ell_3$ to form a geodesic triangle. 

Finally, consider the \lq\lq midpoint\rq\rq vertex $M=[\#^k K_{1,1}]$ on $\ell_3$. By direct computation we find that
\begin{align*}
d[M,\ell_1] & = \min_{0\le m\le k} d[M,[\#^mK_{0,1}]) \ge \min_{0\le m\le k} |\tau (\#^k K_{1,1})-\tau(\#^mK_{0,1})| = k, \cr  
d[M,\ell_2] & = \min_{0\le m\le k} d[M,[L\#(\#^mK_{1,1})]) \ge \min_{0\le m\le k} |\tau (\#^k K_{1,1})-\tau(L\#(\#^mK_{1,1}))| =|k-m|,\cr
d[M,\ell_2] & = \min_{0\le m\le k} d[M,[L\#(\#^mK_{1,1})]) \ge \min_{0\le m\le k} |s' (\#^k K_{1,1})-s'(L\#(\#^mK_{1,1}))| =m.
\end{align*}
Since $m$ ranges from $0$ to $k$, we find that 
\[
	d(M,\ell_1) \ge k \quad \text{ and } \quad d(M,\ell_2)\ge k/2.
\]
Given any $\delta\ge 0$ and picking $k>2\delta$ shows that the geodesic triangle $\ell_1\cup\ell_2\cup \ell_3$ in $\mathcal{CK}_{\slashoverback}$ is not $\delta$-thin. We summarize our finding in the next proposition.

\begin{proposition}\label{PropositionOnDeltaHyperbolicityOfConcordanceGraphs}
For every $\delta\ge 0$ there exists a geodesic triangle in the concordance knot graph $\mathcal{CK}_{\slashoverback}$ that is not $\delta$-thin. Accordingly, $\mathcal{CK}_{\slashoverback}$ is not $\delta$-hyperbolic for any $\delta \ge 0$. 	
\end{proposition}
This proves the portion of Theorem \ref{MainTheorem1} concerning concordance Gordian graphs, and thus completes the proof of said theorem. 

\begin{remark}
We were not able to prove that the concordance knot graphs $\mathcal{CK}_{H(n)}$ are not $\delta$-hyperbolic, even for the base case of $n=2$. This is chiefly because we do not know of two  ``independent" lower bounds on the metric $d$ on  $\mathcal{CK}_{H(n)}$, akin to the roles played by $\tau$ and $s'$ for the case of $\mathcal{CK}_{\slashoverback}$. 
\end{remark}

%%%%%%%%%%%%%%%%%%%%%%%%%%%%%%%%%%%%%%%%%%%%%%%%%%%%%%%%%%%%%
%%%%%%%%%%%%%%%%%%%%%%%%%%%%%%%%%%%%%%%%%%%%%%%%%%%%%%%%%%%%%
\section{Quotient Knot Graphs}
\label{quotient triangles}

This section is devoted to the study of quotient knots graphs as introduced in Section \ref{SectionIntroduction}  and more precisely defined in Section \ref{knot graphs} (cf. Definition \ref{DefinitionOfKnotGraphs}). We study two types of quotient knot graphs, those resulting from the use of a single unknotting operation and a single knot invariant (Section \ref{quotients 1}),  and those obtained from a single unknotting operation in combination with two knot invariants (Section \ref{quotients 2}). Examples \ref{CrossingAndUnkottingNumber}--\ref{ExampleCCAndS/2} provide the isometry types for the quotient graphs $Q\mathcal K_{\slashoverback}^{u}$, $Q\mathcal K_{H(2)}^{\gamma_4}$, $Q\mathcal K_{\slashoverback}^{u}$, $Q\mathcal K_{\slashoverback}^\tau$ and  $Q\mathcal K_{\slashoverback}^{s/2}$ respectively.  Example \ref{isometry} determines the isometry class of $Q\mathcal K^{\{g_4, u\}}_{\slashoverback}$. Each of these examples is shown to meet the hypotheses of two general results about quotient graphs, Theorems \ref{IsometryBetweenQuotinetKnotGraphAndIntegers} and \ref{IsometryBetweenQuotinetKnotGraphAndIntegersCaseTwo}. 

%%%%
%%%%
\subsection{Quotients with respect to a single unknotting operation and knot invariant}
\label{quotients 1}
%%%
%%%
Let $\mathcal O$ be an unknotting operation on knot diagrams, and let $\mathcal I$ be an integer-valued knot invariant compatible with $\mathcal O$ (Definition \ref{DefinitionOfCompatibility}). 
In the next theorem let $|\cdot|$ denote the Euclidean norm on $\mathbb R$, as well as its restrictions to various subsets of $\mathbb R$.  

\begin{theorem} \label{IsometryBetweenQuotinetKnotGraphAndIntegers}
Let $\mathcal I$ be an integer-valued knot invariant with image $\mathbb N$, $\mathbb N\cup\{0\}$ or $\mathbb Z$, and let $\mathcal O$ be an unknotting operation compatible with $\mathcal I$. Let $Q\mathcal K_\mathcal O^\mathcal I$ be the associated quotient knot graph and let $d$ denote its metric. If for every $n\in \text{Im}(\mathcal I)$ there exists a knot $K_n$ with $\mathcal I(K_n) = n$, and where $K_n$ and $K_{n+1}$ are related by an $\mathcal O$-move or its inverse, then the function 
\[ \mathcal I : (Q\mathcal K_{\mathcal O}^{\mathcal I}, d) \to (\text{Im}(\mathcal I) , |\cdot|) \]
is an isometry.  
\end{theorem} 
%%%
%%%
\begin{proof}
Recall that for a knot $K$ the equivalence class $[K]_\mathcal O^\mathcal I$ consists of all knot $K'$ with $\mathcal I (K') = \mathcal I(K)$. If $[K]_\mathcal O^\mathcal I \ne  [K']_\mathcal O^\mathcal I$ are any two vertices in $Q\mathcal K_\mathcal O^\mathcal I$ with $d([K]_\mathcal O^\mathcal I, [K']_\mathcal O^\mathcal I) = n \ge 1$, let $K_0, \dots, K_n$ be knots such that a single $\mathcal O$-move or its inverse relates $K_i$ to $K_{i+1}$, and such that $K_0\in [K]_\mathcal O^\mathcal I$ and $K_n\in [K']_\mathcal O^\mathcal I$. Write $\mathcal I(K_{i+1}) = \mathcal I (K_i) + \eps_{i+1}$ for some choice of $\eps_{i+1}\in \{-1, 0, 1\}$ (which is possible since $\mathcal O$ and $\mathcal I$ have been assumed to be compatible). Then 
\[ |\mathcal I(K') - \mathcal I(K)| = |\mathcal I(K_n) - \mathcal I(K_0)| = |\eps_1+\dots+\eps_n| \le n = d([K']_\mathcal O^\mathcal I,[K]_\mathcal O^\mathcal I). \]
It follows that 
\[ |n-m| = |\mathcal I (K_n) - \mathcal I (K_m)| \le d([K_n]_\mathcal O^\mathcal I,[K_m]_\mathcal O^\mathcal I). \]
On the other hand, since $K_n\in [K_n]_\mathcal O^\mathcal I$ and $K_m\in [K_m]_\mathcal O^\mathcal I$ and $K_n$ and $K_m$ differ by at most $|n-m|$ $\mathcal O$-moves and/or their inverses, it follows that $d([K_n]_\mathcal O^\mathcal I,[K_m]_\mathcal O^\mathcal I) \le |n-m|$ and hence 
\[ d([K_n]_\mathcal O^\mathcal I,[K_m]_\mathcal O^\mathcal I) = |n-m|, \]
completing the proof of the theorem. 
\end{proof}
%%%
%%%
\begin{example} \label{CrossingAndUnkottingNumber}
In this example take $\mathcal O$ to be a crossing change operation, and take $\mathcal I$ to be the unknotting number $u$, and notice that these are compatible in the sense of Definition \ref{DefinitionOfCompatibility}. For $n\in \mathbb N\cup \{0\} = \text{Im}(\mathcal I)$, let $K_n = T(2n+1,2)$. Then $u(K_n) = n$ and $K_{n+1}$ can be changed into the knot $K_n$ by a single crossing change.  Thus, by  Theorem \ref{IsometryBetweenQuotinetKnotGraphAndIntegers}, there is an isometry between $Q\mathcal K_{\slashoverback}^{u}$ and $\mathbb N\cup\{0\}$.
\end{example}

\begin{example} \label{ExampleH(2)AndGamma4}
Take $\mathcal O$ to be a noncoherent band move (an $H(2)$ move), take $\mathcal I$ to be $\gamma_4$ (the non-orientable smooth 4-genus) and for $n\in \mathbb N = \text{Im}(\mathcal I)$, let $K_n=T(2n+2,2n+1)$. Then by \cite{Batson}, the knots $K_n$ satisfy the assumptions of Theorem \ref{IsometryBetweenQuotinetKnotGraphAndIntegers}, leading to the isometry $Q\mathcal K_{H(2)}^{\gamma_4}  \cong \mathbb N$.
\end{example} 

\begin{example} \label{ExampleCCAndG4}
Taking $\mathcal O$ to be a crossing change operation $\slashoverback$, taking $\mathcal I = g_4$ and for $n\in \mathbb N\cup \{0\}=\text{Im}(g_4)$ letting $K_n=T(2n+1,2)$ satisfies the assumptions of Theorem \ref{IsometryBetweenQuotinetKnotGraphAndIntegers}, giving the isometry $Q\mathcal K_{\slashoverback}^{g_4}\cong \mathbb N\cup \{0\}$. 
\end{example} 

\begin{example}  \label{ExampleCCAndTau}
Taking $\mathcal O$ to be a crossing change operation $\slashoverback$, taking $\mathcal I  = \tau$ (the Ozsv\'ath-Szab\'o tau invariant \cite{OZS1}) and for $n\in \mathbb Z=\text{Im}(\tau)$ letting 

\[ K_n=\left\{
\begin{array}{cl}
T(2n+1,2) & , \quad n\ge 0,  \cr
T(2n-1,2) & , \quad n< 0,
\end{array}
\right.\]

satisfies the assumptions of Theorem \ref{IsometryBetweenQuotinetKnotGraphAndIntegers}. Thus $Q\mathcal K_{\slashoverback}^\tau$ is isometric to $\mathbb Z$.
\end{example} 

\begin{example}  \label{ExampleCCAndS/2}
Taking $\mathcal O$ to be a crossing change operation $\slashoverback$, letting $\mathcal I = s'$ (half of the Rasmussen $s$ invariant \cite{Ras}) and for $n\in \mathbb Z = \text{Im}(s')$ letting 
 
\[ K_n=\left\{
\begin{array}{cl}
T(-2n-1,2) & , \quad n\ge 0,  \cr
T(-2n+1,2) & , \quad n< 0,
\end{array}
\right. \]

satisfies the assumptions of Theorem \ref{IsometryBetweenQuotinetKnotGraphAndIntegers}, proving that the spaces $Q\mathcal K_{\slashoverback}^{s/2}$ and $\mathbb Z$ are isometric. 
\end{example} 

\begin{example}  \label{ExampleForNonCompatibleCase}
This example illustrates that when picking an unknotting operation $\mathcal O$ and knot invariant $\mathcal I$ that are not compatible in the sense of Definition \ref{DefinitionOfCompatibility}, the resulting metric space $Q\mathcal K_{\mathcal O}^{\mathcal I}$ can be very different from what is asserted in Theorem \ref{IsometryBetweenQuotinetKnotGraphAndIntegers}. 

Take $\mathcal O$ to be the operation of noncoherent band moves ($H(2)$ moves) and let $\mathcal I = g_4$. Note that the values of $g_4$ of a pair of knots differing by a single $H(2)$-move may differ by an arbitrarily large amount. The vertices of $Q\mathcal K_{H(2)}^{g_4}$ can be identified with $\mathbb N \cup \{0\} = \text{Im}(g_4)$, under the correspondence $n\mapsto [T(2n+1,2)]_{H(2)}^{g_4}$. For each $n\in \mathbb N$ a single band move renders $T(2n+1,2)$ unknotted, showing that 
\[ d(n,0) = 1, \quad \forall n\in \mathbb N, \]
where $d$ is the induced metric on $Q\mathcal{K}_{H(2)}^{g_{4}}$.

In particular, $d(n,m) \le 2$ for all $n,m \in \mathbb N\cup \{0\}$. There is a band move that transforms $T(2n+1,2)$ into $T(2n-3,2)$ (see for example \cite[Figure 2]{LMV}) showing additionally that 
\[ d(n,m) = 1, \quad \text{ if } |n - m| = 4. \]
These relations don't fully pin down the metric space  $Q\mathcal K_{H(2)}^{g_4}$ but they show that it is not isometric to a subspace of $\mathbb R$. 
\end{example}

%%%%%%%%%%%%%%%%%%%%%%%%%%%%%%%%%%%%%%%%%%%%%%%%%%%%%%%%%%%%%%%%%%%%%%%%%%%%%

\subsection{Quotients with respect to a single operation and two knot invariants}
\label{quotients 2}

 Let $||\cdot||_1$ and $||\cdot ||_\infty$ denote the $\ell_1$- and $\ell_\infty$-norms on $\mathbb R^2$, as well as their restrictions to various subsets of $\mathbb R^2$. 

\begin{theorem} \label{IsometryBetweenQuotinetKnotGraphAndIntegersCaseTwo}
Let $\mathcal O$ be an unknotting operation and let $\mathcal I_1, \mathcal I_2$ be two integer-valued knot invariants compatible with $\mathcal O$. Assume that if  $(m_1,n_1)$ and $(m_2,n_2)$ both lie in the image of $\mathcal I_1 \times \mathcal I_2$, then either 

\begin{equation} \label{EquationConvexnessConitionForTwoInvariants}
\cup_{m, n }  \{(m_1,n), (m,n_2\} \subset  \text{Im}(\mathcal I_1 \times \mathcal I_2) \quad \text{ or }  \quad \cup_{m, n }  \{(m,n_1), (m_2,n\} \subset  \text{Im}(\mathcal I_1 \times \mathcal I_2), 
\end{equation}

with both unions taken over integers $m$ between $m_1$ and $m_2$, and integers $n$ between $n_1$ and $n_2$. Let $\mathbb I = \{\mathcal I_1, \mathcal I_2\}$ and let $Q\mathcal K_\mathcal O^\mathbb I$ be the associated quotient knot graph with metric $d$. Suppose there exists a family $\{K_{m,n} \, |\, (m,n) \in \text{Im}(\mathcal I_1 \times \mathcal I_2)\}$ of distinct knots $K_{m,n}$ such that $\mathcal I_1 (K_{m,n}) = m$, $\mathcal I_2(K_{m,n}) = n$, and such that  

\begin{itemize}
\item $K_{m,n}$ and $K_{m,n+1}$ are related by an $\mathcal O$-move or its inverse, whenever $(m,n)$ and $(m,n+1)$ both lie in the image of $\mathcal I_1\times \mathcal I_2$, and
\vskip1mm
\item $K_{m,n}$ and $K_{m+1,n}$ are related by an $\mathcal O$-move or its inverse, whenever $(m,n)$ and $(m+1,n)$ both lie in the image of $\mathcal I_1\times \mathcal I_2$. 
\end{itemize}  

Then the function 
\[ \mathcal I_1\times \mathcal I_2  : (Q\mathcal K_{\mathcal O}^\mathbb I, d) \to (\text{Im}(\mathcal I_1\times \mathcal I_2) , |\cdot|_1) \]
is bi-Lipschitz and satisfies the inequality 
\[ ||(m_1,n_1), \, (m_2,n_2)||_\infty \le d([K_{m_1,n_1}],[K_{m_2,n_2}])  \le ||(m_1,n_1), \, (m_2,n_2)||_1. \]
\end{theorem} 

\begin{proof}
Given a knot $K$, for simplicity of notation we shall write $[K]$ to mean the equivalence class $[K]_\mathcal O^\mathbb I$. The compatibility assumption between $\mathcal O$ and $\mathbb I$ implies that 
\[ d([K],[K']) \ge |\mathcal I_j (K) - \mathcal I_j (K')|, \]
for $j=1,2$ and for any pair of knots $K, K'$ related by an $\mathcal O$-move or its inverse. From these, in complete analogy with the proof of Theorem \ref{IsometryBetweenQuotinetKnotGraphAndIntegers} (while relying on assumption \eqref{EquationConvexnessConitionForTwoInvariants}), one obtains  
\[ d([K_{m,n_1}], [K_{m,n_2}]) = |n_1-n_2| \quad \text{ and } \quad  d([K_{m_1,n}], [K_{m_2,n}]) = |m_1-m_2|, \]
whenever $(m,n_1), (m,n_2), (m_1,n),(m_2,n)$ lie in the image of $\mathcal I_1 \times \mathcal I_2$.  These two equalities show that there are no $\mathcal O$-moves between knots $K_{m,n_1}$ and $K_{m,n_2}$ if $|n_1-n_2|\ge 2$, and similarly there are no $\mathcal O$-moves between knots $K_{m_1,n}$ and $K_{m_2,n}$ if $|m_1-m_2|\ge 2$.

Suppose there is an $\mathcal O$-move from $K_{m_1,n_1}$ to $K_{m_2,n_2}$ for $m_1\ne m_2$ and $n_1\ne n_2$. Then 
\begin{align*}
1 & = d([K_{m_1,n_1}],[K_{m_2,n_2}]) \ge |\mathcal{I}_{1} (K_{m_1,n_1}) - \mathcal{I}_{1} (K_{m_2,n_2})| = |m_1-m_2|, \cr
1 & = d([K_{m_1,n_1}],[K_{m_2,n_2}]) \ge |\mathcal{I}_{2} (K_{m_1,n_1}) - \mathcal{I}_{2} (K_{m_2,n_2})| = |n_1-n_2|.
\end{align*}
We find that the only such $\mathcal O$-moves possible are the ones connecting a knot $K_{m,n}$ to the knots $K_{m\pm 1, n\pm 1}$ (with both signs chosen arbitrarily). Observe then that the distance $d([K_{m_1,n_1}],[K_{m_2,n_2}])$ is minimized when all possible $\mathcal O$-moves of these types exist. Thus, a lower bound on 
\[ d([K_{m_1,n_1}],[K_{m_2,n_2}]) \] 
is given by 
\[ |(m_1,n_1), (m_2,n_2)|_\infty  = \max\{|m_1-m_2|, \, |n_1-n_2|\} \le d([K_{m_1,n_1}],[K_{m_2,n_2}]). \]
On the other hand, for an arbitrary pair $(m_1,n_1), (m_2,n_2)\in \text{Im}(\mathcal I_1 \times \mathcal I_2)$, it is clear that 
\[ d([K_{m_1,n_1}],[K_{m_2,n_2}]) \le |n_1-n_2| + |m_1-m_2| = ||(m_1,n_1), (m_2, n_1)||_1. \]
This claim relies on assumption \eqref{EquationConvexnessConitionForTwoInvariants}. Indeed, if for instance $\cup_{m, n }  \{(m_1,n), (m,n_2)\} \subset  \text{Im}(\mathcal I_1 \times \mathcal I_2)$ (with $m$ between $m_1$ and $m_2$ and $n$ between $n_1$ and $n_2$), then there are $|n_1-n_2|$ $\mathcal O$-moves that connect $K_{m_1,n_1}$ to $K_{m_1,n_2}$, and a further $|m_1-m_2|$ $\mathcal O$-moves that connect the latter knot to $K_{m_2,n_2}$. A similar argument applies in the case that  $\cup_{m, n }  \{(m,n_1), (m_2,n\} \subset  \text{Im}(\mathcal I_1 \times \mathcal I_2)$.
\end{proof}

\begin{example}
\label{isometry}
\textbf{Isometry type of $Q\mathcal K^{\{g_4, u\}}_{\slashoverback}$.}

Consider $\mathcal O$ to be the crossing change operation $\slashoverback$, and let $\mathcal I_1 = g_4$ and $\mathcal I_2= u$ (where $u(K)$ is the unknotting number of the knot $K$). Since for any knot $K$ one has the bound $g_4(K) \le u(K)$, it follows that the image of $g_4\times u$ is a subset of the second octant of $\mathbb Z^2$. As we shall see, the image is actually equal to said octant. 

Let $K_{0,1}$ and $K_{1,1}$ be the knots
\begin{align*}
K_{0,1} & = 6_1 = \text{Stevedors knot,}\cr
K_{1,1} & = 3_1 = \text{Trefoil knot}.
\end{align*} 
It is well known and easy to verify that $g_4(K_{0,1}) = 0$, $u(K_{0,1})=1$, $g_4(K_{1,1})=1=u(K_{1,1})$. For integers $0\le m\le n$ define $K_{m,n}$ as 
\[ K_{m,n} = (\#^{n-m}K_{0,1}) \# (\#^m K_{1,1}). \]
In the above $\#^0K$ denotes the unknot. Observe that 
\begin{align*}
g_4(K_{m,n}) & = g_4((\#^{n-m}K_{0,1}) \# (\#^m K_{1,1})) \cr
& \le (n-m)g_4(K_{0,1})+mg_4(K_{1,1})\cr 
& = m, 
\end{align*}
and
\begin{align*}
u(K_{m,n}) & = u((\#^{n-m}K_{0,1}) \# (\#^m K_{1,1})) \cr
& \le (n-m)u(K_{0,1})+mu_4(K_{1,1})\cr 
& = n.
\end{align*}
Since 
\[ \tau(K_{m,n}) =  (n-m)\tau (K_{0,1})+m \tau (K_{1,1})  = m \]
and since $|\tau(K)|\le g_4(K)$ for any knot $K$, it follows that $g_4(K_{m,n}) = m$.  Recall that there is a lower bound for the unknotting number given by the minimal number of generators of $H_1(\Sigma(K);\mathbb Z)$ (see \cite[page 690]{Wendt} or \cite{Nakanishi}). That is, $$e_2(K)\le u(K).$$ 
For $K_{m,n}$ one finds that 
\begin{align*}
H_1(\Sigma (K_{m,n});\mathbb Z)  &= \left( \oplus _{i=1}^{n-m} H_1(L(9,7);\mathbb Z) \right) \oplus \left(\oplus _{j=1}^m H_1(L(3,1);\mathbb Z) \right) \\
& = (\mathbb Z_9)^{n-m} \oplus (\mathbb Z_3)^m. 
\end{align*}
The minimal number of generators for this homology group is $n$, implying that $u(K_{m,n}) = n$, and in particular that 
\[
\text{Im}(g_4\times u) = \{(m,n)\in \mathbb Z^2\,|\, 0\le m \le n\}.
\]
This shows that condition \eqref{EquationConvexnessConitionForTwoInvariants} from Theorem \ref{IsometryBetweenQuotinetKnotGraphAndIntegersCaseTwo} applies in the current setting. 

Lastly, the knots $K_{m,n+1} =  (\#^{n+1-m}K_{0,1}) \# (\#^m K_{1,1})$ and $K_{m,n}=(\#^{n-m}K_{0,1}) \# (\#^m K_{1,1})$ are related by a crossing change that unknots one of the $K_{0,1}$ summands of $K_{m,n+1}$. 
Similarly $K_{m+1,n}$ is related to $K_{m,n}$ via the crossing change that unknots one of the $K_{1,1}$ summands of $K_{m+1,n}$.  

It follows then from Theorem \ref{IsometryBetweenQuotinetKnotGraphAndIntegersCaseTwo} that the function 
\[
g_4\times u : (Q\mathcal K_{\slashoverback}^{g_4,u},d) \to (\text{Im}(g_4\times u, |\cdot |_1)
\]
is bi-Lipschitz, and in particular, $(Q\mathcal K_{\slashoverback}^{g_4,u},d)$ is not $\delta$-hyperbolic for any $\delta \ge 0$.

\end{example}
%%%%%%%%%%%%%%%%%%%%%%%%%%%%%%%%%%%%%%%%%%%%%%%%%%%%%%%%%%%%%
%%%%%%%%%%%%%%%%%%%%%%%%%%%%%%%%%%%%%%%%%%%%%%%%%%%%%%%%%%%%%
\section{Hyperbolicity and Homogeneity in Knot Graphs} 
\label{proofs}
This section builds on results from previous sections to provide proofs of the main theorems from the introduction. Specifically, Theorem \ref{MainTheorem1} is restated in greater generality in Theorem \ref{TheoremAboutNonHyperbolicityResults}, with Corollary \ref{coro:hyper} completing its proof. Theorem \ref{TheoremMain2} is proved in Section \ref{TheoremAboutHyperbolicityResults}, and Theorem \ref{TheoremMain3} is established in Section \ref{SectionPfoofOfTheorem3}. The final Section \ref{homogeneity} is devoted to a discussion of homogeneity and links in knot graphs, and furnishes a proof of Theorem \ref{TheoremMain4}.       
%%%%%%%%%%%%%%%%%%%%%%%%%%%%%%%%%%%%%%%%%
%%%%%%%%%%%%%%%%%%%%%%%%%%%%%%%%%%%%%%%%%
\subsection{Proof of Theorem \ref{MainTheorem1}}
\label{SubsectionProofOfMainTheorem1}
%%%
%%%
\begin{theorem} \label{TheoremAboutNonHyperbolicityResults}
For any $\delta \ge 0$ there exists a geodesic triangle that is not $\delta$-thin in 
\begin{itemize}
\item[(i)] The knot graphs $\mathcal K_{H(n)}$, for all $n\ge 2$. 
\item[(ii)] The concordance graph $\mathcal{CK}_{\slashoverback}$. 
\item[(iii)] The quotient knot graph $Q\mathcal K^{\{g_4, u\}}_{\slashoverback}$. 
\end{itemize}
Accordingly, these graphs are not $\delta$-hyperbolic for any $\delta\ge 0$, and therefore not Gromov hyperbolic. 
\end{theorem}

\begin{proof}
\label{Sec:nonhyperbolicity}
Proposition \ref{PropositionOnDeltaHyperbolicityOfHnGordianGraphs}  from Section \ref{Hn triangles} established that the knot graphs $\mathcal K_{H(n)}$ for all $n\geq 2$ contain geodesic triangles that are not $\delta$-thin for each $\delta\ge 0$. It follows that $\mathcal K_{H(n)}$ is not $\delta$-hyperbolic for any $\delta\ge 0$ and any $n\geq 2$, proving Part (i). 

Similarly, Proposition \ref{PropositionOnDeltaHyperbolicityOfConcordanceGraphs} from Section \ref{Concordance triangles} established this same result for the concordance graph $\mathcal{CK}_{\slashoverback}$, thereby proving Part (ii). Part (iii) was established in Example \ref{isometry}. 
\end{proof}

\begin{corollary}
\label{coro:hyper}
The knot graph $\mathcal{K}_{\slashoverback}$ is not Gromov hyperbolic.
\end{corollary}

\begin{proof}
It suffices to construct a geodesic triangle in $\mathcal{K}_{\slashoverback}$ which is not $\delta$-thin for any $\delta\geq 0$. We shall reuse here the triangles and notation from Section \ref{Concordance triangles}. Thus, consider the triangle in $\mathcal{K}_{\slashoverback}$ with vertices the unknot $U$, $\#^{k} K_{0, 1}, L\#(\#^{k} K_{1, 1})$ and edges $\ell_{1}, \ell_{2}, \ell_{3}$ as in Section \ref{Concordance triangles}. We first claim that this is a geodesic triangle in $\mathcal{K}_{\slashoverback}$, just as it was in $\mathcal{CK}_{\slashoverback}$ back in Section \ref{Concordance triangles}. Note that for a pair of vertices $\#^{m} K_{0, 1}$ and $\#^{n} K_{0, 1}$ in the edge $\ell_1$ with $m, n\in \{0, \cdots, k\}$, 
\[
	d_{\slashoverback}(\#^{m} K_{0, 1}, \#^{n} K_{0, 1})\leq |m-n|,
\]
where $d_{\slashoverback}$ is the induced metric in $\mathcal{K}_{\slashoverback}$. 
On the other hand, 
\[
	d_{\slashoverback}(\#^{m} K_{0, 1}, \#^{n} K_{0, 1})\geq d([\#^{m} K_{0, 1}], [\#^{n} K_{0, 1}])\geq |m-n| 
\]
where $d$ is the  metric in the  concordance graph $\mathcal{CK}_{\slashoverback}$. Hence $d_{\slashoverback}(\#^{m} K_{0, 1}, \#^{n} K_{0, 1})=|m-n|$, and $\ell_{1}$ is a geodesic. By using a similar argument, we can prove that $\ell_{2}, \ell_{3}$ are also geodesics and the geodesic triangle in $\mathcal{CK}_{\slashoverback}$ is also a geodesic triangle in $\mathcal{K}_{\slashoverback}$. It is also not hard to see that 
$$d_{\slashoverback}(\#^{k} K_{1, 1}, \ell_{1}\cup \ell_{2})\geq k/2$$ 
by using a similar argument. Hence, $\mathcal{K}_{\slashoverback}$ is not $\delta$-hyperbolic for any $\delta\geq 0$, and therefore not Gromov hyperbolic. \end{proof}
%%%%%%%%%%%%%%%%%%%%%%%%%%%%%%%%%%%%%%%%5
%%%%%%%%%%%%%%%%%%%%%%%%%%%%%%%%%%%%%%%%%
\subsection{Proof of Theorem \ref{TheoremMain2}}

\begin{theorem} \label{TheoremAboutHyperbolicityResults}
Each of the quotient knot graphs 
\[
Q\mathcal K_{\slashoverback}^{u},  \qquad  Q\mathcal K_{H(2)}^{\gamma_4}, \qquad  Q\mathcal K_{\slashoverback}^{g_4}, \qquad  Q\mathcal K_{\slashoverback}^{\tau}, \qquad  \text{ and } \qquad Q\mathcal K_{\slashoverback}^{s/2}, 
\]
is $\delta$-hyperbolic for any $\delta \ge 0$. Specifically, the first three spaces are isometric to $\mathbb N\cup \{0\}$, while the second two are isometric to $\mathbb Z$, each equipped with the Euclidean metric.  
\end{theorem}
%%%
Theorem \ref{TheoremAboutHyperbolicityResults} generalizes Theorem \ref{TheoremMain2}, and is a direct consequence of Examples \ref{CrossingAndUnkottingNumber}--\ref{ExampleCCAndS/2} respectively, from Section \ref{quotient triangles}.   

\begin{remark}
	
	The results presented in Theorems \ref{TheoremAboutNonHyperbolicityResults} and \ref{TheoremAboutHyperbolicityResults} stand in stark contrast to one another, representing opposite extremes on the \lq\lq $\delta$-hyperbolicity scale\rq\rq. It would appear that hyperbolicity in quotient knot graphs emerges only when the set of invariants $\mathbb I$ used in its construction consists of a single knot invariant, and when that knot invariant is compatible with all the unknotting operations in $\mathbb O$. Indeed, in such a case we find the resulting quotient knot graph to be quasi-isometric to a subset of $\mathbb R$ (cf. Theorem \ref{IsometryBetweenQuotinetKnotGraphAndIntegers}). In all other cases we find that hyperbolicity is absent from knot graphs.   
\end{remark}

\begin{question}
Does there exist a knot graph that is $\delta$-hyperbolic for some, but not all $\delta>0$? 
\end{question}
%%%%%%%%%%%%%%%%%%%%%%%%%%%%%%%%%%%%%%%%%%%
%%%%%%%%%%%%%%%%%%%%%%%%%%%%%%%%%%%%%%%%%%%
\subsection{Proof of Theorem \ref{TheoremMain3}} \label{SectionPfoofOfTheorem3}
The proof of Theorem \ref{TheoremMain3} rests on the observation that there exist knots whose $s'$ and $\tau$-invariants differ, something already exploited in Section \ref{Concordance triangles}. This observation allows us to construct a vertex in the Gordian knot graph that is of unbounded distance from any arbitrary connected sum of torus knots. We do so now. 
\vskip3mm

In a proof by contradition of Theorem \ref{TheoremMain3}, suppose that there exists some universal bound $n$ such that for all knots $K$ there is a connected sum of some number torus knots $T$ with $d(K, T)\leq n$. Because both $s'$ and $\tau$ change by either $-1, 0$ or $1$ under a crossing change, the difference $r = s'-\tau$ provides the lower bound $2d(K, K') \geq | r(K) - r(K')|$ on the  Gordian distance, while on the other hand for any connected sum $T$ of torus knots, $r(T)=0$. Let $K := \#^{2n+2} K_{0, 1}$. Since $\tau(K_{0,1})=0$ and $s'(K_{0,1})=1$, we obtain that $2d(K, T)\geq |r(K)-r(T)| \geq 2n+2$, and so $d(K, T)>n$. Theorem \ref{TheoremMain3} follows.

\begin{remark}
\label{squeeze}
In fact, Theorem \ref{torus ball} could be stated more generally by replacing the set of connected sums of torus knots with the larger set consisting of connected sums of knots whose $s'$ and $\tau$-invariants agree. Indeed, Feller-Lewark-Lobb  define a class of knots called \emph{squeezed knots}, which occur as a slice of a minimal-genus cobordism between positive and negative torus knots \cite{FLL}. Squeezed knots contain torus knots, positive and quasi-positive knots, negative and quasi-negative knots, alternating and homogeneous knots and is closed under connected sums. Such knots have the property that their evaluations on different slice-torus invariants are identical, and in particular, $s'$ and $\tau$ will agree. Any subset of squeezed knots could replace the torus knots in the statement of Theorem \ref{TheoremMain3}. 
\end{remark}

%%%%%%%%%%%%%%%%%%%%%%%%%%%%%%%%%%%%%%%%%%%%%%%%%%%%%%%%%%%%%
%%%%%%%%%%%%%%%%%%%%%%%%%%%%%%%%%%%%%%%%%%%%%%%%%%%%%%%%%%%%%
\subsection{Homogeneity and Links, and the Proof of Theorem \ref{TheoremMain4}}
\label{homogeneity}

Recall that a metric space $(X,d)$ is {\em homogeneous} if for every $x,y\in X$ there exists an isometry $\psi:X\to X$ with $\psi(x) = y$, i.e. if the isometry group of $X$ acts transitively on $X$. If a metric space $(G,d)$ arises from a graph $G$ all of whose edges have length 1, let us define the {\em link of a vertex $v\in Vert(G)$}, denoted $\ell k(v)$, as the induced subgraph of $G$ generated by the set 
\[
	\{ w \in Vert(G) \, |\, d(v,w) = 1\}.
\]
Note that $v\notin \ell k(v)$ and that for $w, u \in \ell k(v)$, an edge $e=\{w,u\}$ belongs to $\ell k(v)$ if and only if $d(w,u)=1$. The {\em diameter of $\ell k(v)$} is the supremum of  $\{d(w,u)\,|\, w,u \in Vert(\ell k(v))\}$.  If $(G,d)$ is a  homogeneous metric space, clearly the links of any pair of vertices are isometric. 

\begin{question}
With regards to the above definition, we ask:
\begin{itemize}
\item[(i)] In which, if any, knot graphs is the link of the (class of the) unknot connected?
\item[(ii)] If the link of the (class of the) unknot is connected, determine if its diameter is finite. If the diameter is finite, calculate or estimate its value.   
\item[(iii)] Which, if any, knot graphs from Definition \ref{DefinitionOfKnotGraphs} are homogeneous?
\end{itemize} 
\end{question}
%%%
%%%
Some of these questions are inspired by the work \cite{HoffmanWalsh} of Hoffman-Walsh which studies the {\em Big Dehn Surgery Graph}. The vertices of this graph are closed orientable 3-manifolds, and edges are formed by 3-manifolds related by a Dehn surgery. Hoffman-Walsh prove that the link of $S^3$ is connected and of finite diameter.  
In another direction,  Nakanishi and Ohyama \cite{NakanishiOhyama} show that the $\#$-Gordian graph $\mathcal K_\#$ is not homogeneous, by utilizing the well understood relation between the Conway polynomial and pass-moves. 

We now state a more detailed version of Theorem \ref{homogeneous}. 
%%%
%%%
\begin{theorem} \label{TheoremOnHomogeneityOfConcordanceGraphs}
Let $\mathbb O = \{\mathcal O_1, \dots, \mathcal O_n\}$ be a collection of unknotting operations and let $\mathcal{CK}_\mathbb O$ be the associated concordance graph. 
\begin{itemize}
\item[(i)] The concordance graph $\mathcal{CK}_{\mathbb O}$ is always homogeneous. Specifically, an isometry of $\mathcal{CK}_\mathbb O$ sending a concordance class $[K]$ to a concordance class $[K']$ is given by 
$$\psi([L]) = [L\#(-\overline{K})\#K'],$$
where $-\overline{K}$ is the reverse mirror of $K$.
%%%
\item[(ii)] The quotient knot graphs $Q\mathcal K^\tau _{\slashoverback}$ and $Q\mathcal K^{s'}_{\slashoverback}$ are homogeneous. 
\end{itemize}  
\end{theorem}
\begin{proof}
Part (ii) of the preceding theorem and the following corollary  are direct consequences of Theorem \ref{TheoremAboutHyperbolicityResults}. 

%%%%%%%%%%%%%%%%%%%%%%%%%%%%%%%%%%%%%%%%%%%%%%%%%%%%%%%%%%%%%
%%%%%%%%%%%%%%%%%%%%%%%%%%%%%%%%%%%%%%%%%%%%%%%%%%%%%%%%%%%%%

Let $\mathbb O = \{\mathcal O_1, \dots, \mathcal O_n\}$ be a collection of distinct unknotting operations, let  $\mathcal {CK}_\mathbb O$ be the associated concordance graph, and let $d$ denote its induced metric. For a fixed pair of knots $K, K'$, let $\psi_{K,K'} :\mathcal {CK}_\mathbb O \to \mathcal {CK}_\mathbb O$ be the function 
\[
	\psi([L]) = [L\#(-\overline{K})\#K'],
\]
where $-\overline{K}$ denotes the reverse mirror of $K$. 
Note that $\psi([K]) = [K']$ and that $\psi_{K,K'}$ is a bijection with inverse $\psi_{K',K}$. 

To show that $\psi_{K,K'}$ is an isometry of $\mathcal{CK}_\mathbb O$, let $[L]$ and $[L']$ be concordance classes with $d([L],[L']) = 1$. Without loss of generality we may assume that the knots $L$ and $L'$ are related by an $\mathcal O_i$-move (or its inverse) for some $i\in \{1,\dots, n\}$. It follows that the knots $L\#(-\overline{K})\#K'$ and $L'\#(-\overline{K})\#K'$ are also related by an $\mathcal O_i$-move (or its inverse) showing that 
\[
	d(\psi_{K,K'}([L]), \psi_{K,K'}([L']) = d([L\#(-\overline{K})\#K'],[L'\#(-\overline{K})\#K']) = 1.
\] 
Iterating this argument one finds that for any pair of concordance classes $[L]$ and $[L']$ (with arbitrary $d([L],[L'])$) the following inequality holds:
\[ d(\psi_{K,K'}([L]),\psi_{K,K'}([L']))\le d([L],[L']). \]
Repeating the argument for $(\psi_{K,K'})^{-1} = \psi_{K',K}$ one obtains the opposite inequality, showing that $\psi_{K,K'}$ is an isometry. 

\end{proof}

%%%
\begin{corollary}
The link of the class of the unknot in the quotient knot graphs $Q\mathcal K_{\slashoverback}^{g_4}$, $Q\mathcal K_{\slashoverback}^{u}$ and $Q\mathcal K_{H(2)}^{\gamma_4}$ is a singleton set. In the quotient knot graphs $Q\mathcal K_{\slashoverback}^{\tau}$ $Q\mathcal K_{\slashoverback}^{s/2}$, the link of the class of the unknot consists of exactly two points and is disconnected.  
\end{corollary}

\section*{Acknowledgements}
We thank Peter Feller for pointing out Remark \ref{squeeze}, and Cornelia Van Cott for useful comments on an earlier version of this paper. The second author is also grateful to the Max Planck Institute for Mathematics in Bonn for its hospitality and financial support.

%\section*{Conflict of interest}
%%
%The authors declare that they have no conflict of interest.

\bibliographystyle{alpha}
\bibliography{bibliography}

\end{document}